\documentclass[a4paper,12pt]{article}
\RequirePackage{ffwg_sty_basic}

\usepackage[margin=1in]{geometry}
\usepackage{multicol}
\usepackage{color}
\usepackage{xspace}
\usepackage{pdfwidgets}
\usepackage{enumerate}
\usepackage{enumitem}
\usepackage{booktabs}

\begin{document}

\title{
U-match factorization: sparse homological algebra, lazy cycle representatives, and dualities in persistent (co)homology
}

\author{Haibin Hang\footnote{University of Delaware} \and Chad Giust$\text{i}^\ast$ \and Lori Ziegelmeier\footnote{Macalester College} \and Gregory Henselman-Petrusek\footnote{University of Oxford, correspondence: henselmanpet@maths.ox.ac.uk}}

\date{\today}
\maketitle

\newcommand{\cols}{\text{\sc{Cols}}}
\newcommand{\rows}{\text{\sc{Rows}}}
\newcommand{\svec}{e} 

\newcommand{\allr}{\upto{m}} 
\newcommand{\allc}{\upto{n}} 


\newcommand{\field}{\mathbb K}
\newcommand{\binary}{ {\mathbb F}_2}

\newcommand{\simpcomp}{K}
\newcommand{\finitepoints}{X}
\newcommand{\Simplices}[0]{S}
\newcommand{\Chains}{C}
\newcommand{\Chaincomplex}{\mathcal{C}}
\newcommand{\coeffs}{G}
\newcommand{\boundaryop}{\partial}
\newcommand{\simplexdim}{n}
\newcommand{\simplex}{\sigma}
\newcommand{\Cycles}[0]{Z}
\newcommand{\Boundaries}[0]{B}
\newcommand{\Homologies}[0]{H}
\newcommand{\VectorSpace}[0]{V}
\newcommand{\Basis}{\mathcal{B}}


\newcommand{\chaina}{x}
\newcommand{\chainb}{y}

\newcommand{\mcob}{E}
\newcommand{\mbasisR}[1]{{\mathscr E}_{#1}}
\newcommand{\mbasisL}[1]{{\mathscr E}^*_{#1}} 
\newcommand{\rreduction}{Y}
\newcommand{\rreductionAlt}{\mathfrak Y}

\newcommand{\fchx}[1]{F_{#1} \Chains}
\newcommand{\Birth}{\mathrm{Birth}}
\newcommand{\Death}{\mathrm{Death}}
\newcommand{\del}{\mathrm{clear}} 
\newcommand{\low}{\mathrm{low}}
\renewcommand{\pmod}{\mathscr{P}}
\newcommand{\chaindim}[1]{\epsilon_{#1}}  
\newcommand{\fchaindim}[2]{\epsilon_{#2}^{(#1)} }  
\newcommand{\fchaindimRel}[2]{  \theta_{#2}^{(#1)}  }

\newcommand{\fparammin}{0} 
\newcommand{\fparammax}{N} 
\newcommand{\bornby}[1]{^{(#1)}}
\newcommand{\fparam}{p}

\newcommand{\filta}{F}
\newcommand{\filtb}{G}

\newcommand{\relrange}[2]{\Theta_{#2}^{(#1)}}


\newcommand{\DEF}{\mathrm{def}}
\newcommand{\VAL}{\mathrm{val}}

\newcommand{\lata}{\mathfrak M}
\newcommand{\latb}{\mathfrak D}
\newcommand{\subspacelattice}{\mathfrak W}


\renewcommand{\Im}{\mathrm{Im}}
\newcommand{\Ker}{\mathrm{Ker}}
\newcommand{\spanbrac}[1]{\langle #1 \rangle}

\newcommand{\diagmat}{\mathscr D}

\newcommand{\basisa}{A}
\newcommand{\mbasis}{B}

\newcommand{\IMatch}{\mu}
\newcommand{\mRow}{\mathrm{row}}
\newcommand{\mCol}{\mathrm{col}}

\newcommand{\id}{I} 
\newcommand{\pcols}{ \kappa } 
\newcommand{\prows}{ \rho } 
\newcommand{\nprows}{ \bar{\prows} } 
\newcommand{\npcols}{ \bar{\pcols}} 
\newcommand{\prow}{\rho} 
\newcommand{\pcol}{\kappa} 
\newcommand{\prowp}{\prow^*} 
\newcommand{\pcolp}{\pcol^*} 
\newcommand{\nprow}{\bar{\prow} } 
\newcommand{\npcol}{ \bar{\pcol} } 
\newcommand{\npiv}{k} 
\newcommand{\RR}{ \mathscr{R} } 
\newcommand{\Ri}{  \RR^{-1} } 
\newcommand{\Rrr}{\RR_{\prows \prows}}
\newcommand{\Rirr}{(\Rrr)^{-1}}
\newcommand{\CC}{ \mathscr{C} } 
\newcommand{\Ccc}{ \CC_{\pcols \pcols} } 
\newcommand{\Ci}{ \CC^{-1} } 
\newcommand{\MM}{M} 
\newcommand{\Mrk}{ \MM_{\prows \pcols}} 
\newcommand{\Mirk}{ \MM_{\prows \pcols}^{-1}} 
\newcommand{\DD}{D} 
\newcommand{\Drk}{\DD_{\prows \pcols}}
\newcommand{\Drki}{\Drk^{-1}}
\renewcommand{\AA}{\mathscr{A}}
\newcommand{\Ai}{\AA^{-1}}
\newcommand{\matchscale}{\mu}

\newcommand{\row}{\text{\sc Row}} 
\newcommand{\col}{\text{\sc Col}} 

\newcommand{\upto}[1]{ \mathbf{#1} }

\newcommand{\COB}{COMB} 

\begin{abstract}

Persistent homology is a leading tool in topological data analysis (TDA).  Many problems in TDA can be solved via homological -- and indeed, linear -- algebra.  However, matrices in this domain are typically large, with rows and columns numbered in billions. Low-rank approximation of such arrays typically destroys essential information; thus, new mathematical and computational paradigms are needed for  very large, sparse matrices.  

We present the U-match matrix factorization scheme to address this challenge. U-match has two desirable features.  First, it admits a compressed storage format that reduces the number of nonzero entries held in computer memory by one or more orders of magnitude over other common factorizations.  Second, it permits direct solution of diverse problems in linear and homological algebra, without decompressing matrices stored in memory.  These problems include look-up and retrieval of rows and columns; evaluation of birth/death times, and extraction of generators in persistent (co)homology; and, calculation of bases for boundary and cycle subspaces of filtered chain complexes.  Such bases are key to unlocking a range of other topological techniques for use in TDA, and U-match factorization is designed to make such calculations broadly accessible to practitioners.  

As an application, we show that individual cycle representatives in persistent homology can be retrieved at time and memory costs orders of magnitude below current state of the art, via global duality.  Moreover, the algebraic machinery needed to achieve this computation already exists in many modern solvers.

\tb{Key words} persistent homology, topological data analysis, matrix reduction, algorithm engineering

\end{abstract}

\newcommand{\EE}{E}

\newcommand{\homology}{\mathrm{H}}
\newcommand{\betti}{\beta}
\newcommand{\topspace}{X}
\newcommand{\hdegree}{n}
\newcommand{\barcode}{\mathrm{Barcode}}
\newcommand{\cell}{\sigma}

\newcommand{\intmax}{m}

\newcommand{\pmodsummanda}{U}
\newcommand{\pmodsummands}{\mathscr U}
\newcommand{\pmodsummandaa}{W}
\newcommand{\pmodsummandss}{\mathscr W}

\newcommand{\cyclea}{z}
\newcommand{\cocyclea}{z^*}

\newcommand{\fparamb}{q}

\newcommand{\Ralt}{\mathfrak R}
\newcommand{\Valt}{\mathfrak V}

\section{Introduction}

Persistent homology provides topological summaries of data across a hierarchy of scales.  To obtain such a summary, one typically executes a series of refinements: first, data is transformed  from a raw format into a filtered chain complex; second, the complex is transformed into a sequence of vector spaces and linear maps called a persistence module; third and finally, the module is transformed into a  multiset of intervals called a barcode.

Barcodes find diverse applications in data science, but information is lost at each stage of refinement.  
For example, to associate an individual interval with a specific topological feature, one must generally have knowledge of the persistence module.  
To implement an algorithm that determines the range of filtration parameters where an arbitrary cycle, $\zeta$  represents a nontrivial homology class $[\zeta]$, one must generally know something about the underlying filtered complex.  

These examples underscore the point that lower-level objects  -- persistence modules and filtered chain complexes -- are essential to developing full-fledged models of scientific data.  However, two challenges  limit our ability to work with such objects effectively:

\begin{enumerate}
    \item \underline{\smash{Obtaining}} linear bases.  Linear algebra can solve a variety of important problems in homological persistence (c.f.\ \S\ref{sec:exampleproblems}), however few high-performance persistent homology libraries return bases for a persistence module.  Fewer still give  access to bases for filtered chain complexes, and none, to our knowledge, gives convenient access to bases for subspaces of (relative) (co)cycles and (co)boundaries.  Without such data, computer algebra has no basis on which to operate.
    \item \underline{\smash{Storing and using}} linear bases to solve problems.  Chain complexes used in TDA  have cycle and boundary spaces of very high dimension --  often in the billions or hundreds of billions.  It is impractical to store a column basis for such a space in the form of a sparse CSC or triplet matrix.  Thus, new data structures are needed. 
\end{enumerate}

We address these challenges by a lazy\footnote{A \emph{lazy} method relies on efficient just-in-time computations to avoid storage of excessively large data, such as boundary matrices for chain complexes.} strategy that combines two fundamental ingredients in modern persistent homology computation,             \emph{algebraic symmetry}, expressed by a notion of global duality (Theorem \ref{thm:globalduality}), and           \emph{computational asymmetry}, specifically a large disparity in computation time for persistent homology versus cohomology.
 The strategy uses a matrix decomposition scheme which we call U-match to simplify much of the machinery involved.
    
    Before we begin, we will briefly unpack some of the essential ideas; the reader may wish to refer to \cite{de2011dualities} for further background detail.
  \medskip

\subsection{Persistence reviewed}  
  In order to apply persistent homology, data is encoded as a chain complex along with a filtration that describes its hierarchical structure,
 \begin{align*}
        \fchx{} :
            \quad
            0
            =
            \fchx{\fparammin}
            \subseteq
            \cdots
            \subseteq
            \fchx{\fparammax}
            =
            \Chains            
            .
    \end{align*}
   Passing to homology with coefficients in a fixed field $\field$, we obtain a sequence of vector spaces and linear maps induced by the inclusions
    \begin{align*}
        \pmod_\ast( \fchx{} ):
        \quad
        \Homologies_\ast( \fchx{\fparammin} )
        \to
        \cdots
        \to
        \Homologies_\ast( \fchx{\fparammax} )
    \end{align*}
This sequence is called the \emph{(graded) homological persistence module} for $\fchx{}$, or simply \emph{graded PH module} for short.

It is standard to view the graded PH module as a  quiver representation of the directed graph $\fparammin \to \cdots \to \fparammax$.  A celebrated theorem of Gabriel \cite{gabriel1972unzerlegbare} states that every such representation decomposes as a direct sum of indecomposable pure-graded\footnote{By definition, an element $v$ of a graded vector space $V = \bigoplus_p V^p$ is \emph{pure-graded} if $v \in V^p$ for some $p$.} sub-representations.  Concretely, this means that for each homological grading $\simplexdim$
    \begin{align}
        \pmod_\simplexdim( \fchx{} )
        =
        \bigoplus_{V \in \mathscr{V}} V
        \label{eq:pmoddecomposition}
    \end{align}
for some family of submodules $\mathscr{V}$, where each $V = (V^\fparammin \to \cdots \to V^\fparammax) \in \mathscr{V}$ is isomorphic to an \emph{interval representation} of the form $$0\to  \cdots \to  0 \to \field \xrightarrow[]{\cong} \cdots \xrightarrow{\cong} \field \to 0 \rightarrow \cdots \to 0.$$  The \emph{support} of a nonzero  interval representation $V$ is $\supp(V)~=~\{~i~:~V^i~\neq~0\}$,  commonly written as a half-open interval $[a,b)$, where $a = \min \supp(V)$ is the index such that $V^a \neq 0$ and $V^{a-1}=0$; and $b = 1+\max \supp(V)$ is the index such that $V^b = 0$ and $V^{b-1} \neq 0$.  

\medskip

Three related objects derived from this decomposition are paramount to the present work. The first is the \emph{barcode} of Equation \eqref{eq:pmoddecomposition}, the indexed family of half-open intervals $(\supp(V))_{V \in \mathscr{V}}$.  It can be shown that all interval decompositions of a given persistence module yield the same barcode in each grading, up to re-indexing.  Thus  $\pmod_\ast( \fchx{} )$ uniquely determines a \emph{multiset} of half-open intervals, called the barcode of $\pmod_\ast(\fchx{} )$, and in each dimension $\simplexdim$ we have a dimension-$\simplexdim$ homological barcode of $\fchx{}$. Barcodes are among the leading shape descriptors used in TDA.  They are complete invariants of the associated sequence of homology maps, meaning that they determine the sequence up to isomorphism.   They are efficient to store in memory, represent visually, and there are a variety of methods for vectorizing them for statistical or machine learning applications.

The second, containing finer information, is a \emph{persistent homology (PH) cycle basis}. This is a subset $T \su  \Cycles \su \Chains$ of the cycle space $\Cycles$ containing one pure-graded cycle representative for each interval in the barcode.  It can be regarded as a solution to the inverse problem of mapping bars back to features of the space itself; it is strictly more informative than the barcode, but (unlike the barcode) it is not uniquely defined.
            
The third and final object, a superset of the PH cycle basis $T \su T'$, is a linear basis for $\Cycles$ that contains pure-graded bases for both the cycles and the boundaries of each space in the filtration.  So far as we are aware no name has been assigned to such a basis, and we will refer to it as a \emph{saecular basis}, c.f. \cite{saecular}. This is the most informative of the three and can be used to solve a wide range of problems, but like the PH cycle basis, it is not unique. 
        
While standard persistent homology occupies a central role in much of TDA literature, there are several natural variants that provide utility in understanding data. These variants are  obtained by replacing  $\Homologies_\ast(\Chains)$ with relative homology $\Homologies_\ast(\Chains, \filta_\fparam \Chains)$, cohomology    $\Homologies^\ast(\Chains) $, or relative cohomology $\Homologies^\ast(\Chains, \filta_\fparam \Chains),$ and deriving analogous sequences of vector spaces and maps. The barcode of any one of these sequences uniquely determines the others, a fact known as global duality for barcodes \cite{de2011dualities}. Of particular interest will be the \emph{persistent relative cohomology (PrcH) cocycle basis}, obtained from the interval decomposition of the relative cohomological persistence module for $\Homologies^\simplexdim(\Chains, \filta_\fparam \Chains)$.
    
\medskip  

\subsection{Algebraic symmetry (global duality)}
  \emph{Algebraic symmetry}, the first of our two ingredients, refers to the notion of \emph{global duality} introduced in \cite{de2011dualities}.  We give a brief description here, and the reader is referred to \cite{de2011dualities} for further details.
  
 Before we begin, we need to set some notation. To avoid an explosion of subscripts in the paper, we use square braces to index individual elements in a matrix $\DD$ or vector $v$, e.g. $\DD[i,j]$ or $v[i]$, and denote by $\row_i(\DD)$ and $\col_i(\DD)$ the $i$th row or column of a matrix $\DD$, respectively.
 
Let $\fchx{}$ be a filtered chain complex, and make the simplifying assumption that each filtration level differs from the last by the introduction of a single new basis vector; that is, $\filta_\fparam \Chains$ is generated by linearly independent chains $\{\cell_{1}, \ldots, \cell_\fparam\}.$ In this basis, the total boundary matrix $D \in \field^{\fparammax \times \fparammax}$ for $\fchx{}$ is given by the equation $\partial \cell_\fparam = \sum_{\fparamb < \fparam} \DD[\fparamb, \fparam] \cell_\fparamb$. 
We say that $\DD$
        is \emph{$(-1)$-graded}, in the sense that $\dim(\cell_i) = \dim(\cell_j)-1$ whenever $\DD[i,j] \neq 0.$
    
There is an important relationship between the PH and PrcH cycle bases for $\fchx{}$ and $\DD,$ given by the following theorem. The part of Theorem \ref{thm:globalduality} that applies to PH was proved in \cite{cohen2006vines}; the dual component for PrcH was introduced in \cite{de2011dualities}. A matrix with exactly one non-zero entry in each row and each column is a \emph{(generalized) matching matrix}\footnote{A \emph{matching matrix} is a generalized matching matrix with entries in $\{0,1\}$. We deal only with generalized matching matrices in this text and  will, by abuse of terminology, drop the qualifier ``generalized.''.}.
        \begin{theorem}[Global duality for generators, \cite{cohen2006vines, de2011dualities}]
        \label{thm:globalduality}
                Suppose that $\EE$ is a 0-graded invertible upper triangular matrix  and $\MM = \EE^{-1}  \DD \EE$ is a matching matrix.   Then,  
                \begin{enumerate}
                    \item The columns of $\EE$ contain a pure-graded PH cycle basis (for homology in each dimension)
                    \item The rows of $\EE^{-1}$ contain a pure-graded PrcH cocycle basis (for relative cohomology in each dimension).
                \end{enumerate}
            In greater detail, for each pair $(i,j)$ with $\MM[i,j] \neq 0$ and each  $k$ so that $\row_k(\MM) = 0$ and $\col_k(\MM) = 0$, let 
            \begin{align*}
                \tau_{ij} &= \col_i(\DD \EE) &
                \omega_{k} &= \col_k(\EE) 
                \\
                \tau^{ij} &= \row_j(\EE^{-1}\DD) &
                \omega^{k} &= \row_k(\EE^{-1}) 
            \end{align*}      
        
    Then,
        \begin{enumerate}
            \item Each $\tau_{ij}$ represents a pure-graded homology class that ``lives'' at each $t \in [i,j)$ in persistent homology; each $\omega_k$ represents a pure-graded homology class that lives at each $t \in [k, \infty)$.  The set of all $n$-cycles $\tau_{ij}$ together with all $n$-cycles $\omega_k$ forms a PH cycle basis.
            \item Each $\tau^{ij}$ represents a pure-graded relative cohomology class that ``lives'' at each  $t \in (i, j]$ in persistent relative cohomology; each $\omega^k$ represents a pure-graded relative cohomology class that lives at each $t \in (-\infty,  k]$. The set of all relative $n$-cocycles $\tau^{ij}$ together with all relative $n$-cycles $\omega^k$ forms a PrcH cocycle basis.            
        \end{enumerate}
            
        \end{theorem}
        
        To the best of our knowledge, there exists no technical term for the matrix $E$ described in Theorem \ref{thm:globalduality}.  For concreteness we give it a name: a \emph{(filtered, pure-graded) Jordan basis} for the total differential boundary matrix $\DD$ is an invertible upper-triangular 0-graded matrix $\EE$ such that $\EE^{-1} \DD \EE$ is a  matching matrix.

        \begin{remark}
        The columns of $\EE$ and rows of $\EE^{-1}$ are useful for much more than computing PH and PrcH basis:
            \begin{enumerate}
                \item One can extract bases for related constructions called persistent relative homology and  persistent cohomology.
                \item One can compute pure-graded bases for the space of (co)cycles and (co)boundaries of $\filta_\fparam \Chains$ for each $\fparam$.
                \item Given an $\simplexdim$-dimensional boundary $b$, one can efficiently calculate the smallest $\fparam$ such that $b \in \filta_\fparam\Boundaries$, the space of boundaries.  More importantly, one can compute an explicit bounding chain $c \in \fchx{p}$.
            \end{enumerate}
        \end{remark}

 \medskip
 
\subsection{Computational asymmetry}
\label{subsec:computationalassymetry}

 There are important scenarios in computation where the practitioner may achieve a predetermined goal in one of two ways: either by applying an algorithm to a matrix $\DD$, or by applying the same algorithm to the anti-transposed\footnote{The antitranspose of $\DD$ is obtained by transposing, then reversing the order of rows and columns. Thus, for example, the first row of $\DD^\perp$ corresponds to the last column of $\DD$.} matrix $\DD^\perp$.
 For example, to compute the rank of $\DD$, one could apply Gaussian-Jordan column reduction either to $\DD$ or to $\DD^\perp$; both methods yield the rank of the matrix.
However, there are cases where one route consumes significantly more time and memory than the other.  We refer to this phenomenon as \emph{computational asymmetry}.
    
Computational asymmetry is endemic to one of the most important algorithms in topological analysis, the so-called \emph{standard algorithm} for persistent homology.  This algorithm is a constrained form of Gaussian-Jordan elimination, and it yields a triplet of matrices $(R, D, V)$ such that  $R = DV$ where 

    \begin{enumerate}
        \item \label{item:rdv1} $V$ is invertible and upper triangular, and 
        \item \label{item:rdv2} $R$ is ``reduced'' in the sense that every nonzero column has its lowest nonzero entry in a distinct row.  Formally, let $\low_R(c): = \max \{ r : \col_c(R)[r] \neq 0 \}$ denote the greatest row index $r$ such that column $c$ has a nonzero entry in row $r$.  Then  $\low_R(c) \neq \low_R( \tilde c )$ for any two distinct nonzero columns $c \neq \tilde c$.
    \end{enumerate}
     This type of decomposition is typically called ``$R = DV$'', however to avoid conflicts in notation we will refer to the triple $(R, D, V)$ as a \emph{right-reduction}.  
    
    Right-reduction enables the computation of pure-graded PH cycle bases and PrcH cocycle bases via the following result of \cite{cohen2006vines, de2011dualities}.  For convenience, given a right-reduction $\rreduction$, write $\mbasisR{\rreduction}$ for the invertible upper triangular matrix such that 
    \begin{align*}
    \col_j( \mbasisR{\rreduction} ) 
    = 
    \begin{cases}
        \col_i(R) & j = \low_R(i) \text{ for some } i \\
        \col_j(V) & else
    \end{cases}
    \end{align*}
    \begin{theorem}[Dual generators and right-reduction]
    \label{thm:globalRDV}
    Suppose that
    \begin{align*}
        \underbrace{R = D V}_{\rreduction}
         &&
        \underbrace{\Ralt = \DD^\perp \Valt}_{\rreductionAlt}
    \end{align*}
    are right-reductions.  Then $\mbasisR{\rreduction}$ and $(\mbasisR{\rreductionAlt}^{\perp})^{-1}$ are (filtered, graded) Jordan bases of $\DD$.  In particular, Theorem \ref{thm:globalduality} applies to both $\mbasisR{\rreduction}$ and $(\mbasisR{\rreductionAlt}^{\perp})^{-1}.$
    \end{theorem}

In light of Theorem \ref{thm:globalRDV}, the practitioner may choose to compute persistent homology in one of two ways: either apply the standard algorithm to $\DD$, or apply it to $\DD^\perp$.  However,  the former consistently consumes \underline{\smash{orders of magnitude more time and memory}} than the latter, in many real world applications.\footnote{In particular, this asymmetry arises in applications of the standard algorithm to filtered clique complexes \cite{de2011dualities, otter2017roadmap, bauer2017phat}; the asymmetry is less pronounced for other topological models, e.g.\ cubical complexes \cite{bauer2017phat}.    These results are corroborated by experiments in the preset work, c.f. \S\ref{sec:experiments}.}  This is the  sense in which calculating  persistent homology via the standard algorithm is computationally asymmetric.  

It is important to note that {barcodes} for filtered complexes can be extracted from either $R$ or $\Ralt$ with very little effort; in fact, for this computation one only needs to know one of  two partially defined functions: $\low_R$ or $\low_\Ralt$.  Either of these can be found by visual inspection of the sparsity pattern of $R$ or $\Ralt$.  In practice, many barcode solvers also ``throw away'' 
the columns of $V$ or $\Valt$ that correspond to zero columns of $R$, since these lie effectively outside the domain of the $\low{}$ function.
Discarding these columns from $V$ as soon as they are computed can save a tremendous amount of computer memory, c.f. \S\ref{sec:experiments}.   

As compared to the problem of computing a barcode, the problem of computing {Jordan bases} is 
highly sensitive to the choice of $R = DV$ versus $\Ralt = \DD^\perp \Valt$ decomposition.  On the one hand, extracting a Jordan basis $\mbasisR{\rreduction}$ from $\rreduction$ is primarily a matter of reading-off some columns of $R$ and $V$. 
On the other hand,  extracting a Jordan basis $(\mbasisR{\rreductionAlt}^{\perp})^{-1}$ from $\rreductionAlt$  involves the much more intensive process of inverting the anti-transpose of a large matrix.   

Thus the practitioner faces not one, but {two} significant  asymmetries:(i) it is  easier to compute $\rreductionAlt$ than $\rreduction$, and (ii)  it is  easier to extract a Jordan basis from $\rreduction$ than from $\rreductionAlt$.  The latter poses no obstacle for barcodes, but interlocks with the former to make Jordan bases very challenging to compute, in general.  

This disparity 
has been reflected both in algorithm development and in software implementation over the past decade.  While efficient barcode computation has surged ahead in both arenas (see \S\ref{sec:literature}), few software packages to this day even offer the option to return a PH cycle basis.  

This dearth of basis solvers cannot be attributed entirely to the difficulty of computing a basis, however.  An equally important question concerns what, precisely, the user might be able to do with such a basis, if they had one.  Barcodes require little memory to store, and are easy to analyze on essentially any computer platform.  Bases become interesting only when paired with a computer system that can perform sparse algebraic operations (matrix multiplication, manipulation of sparse vectors, etc.).  Most of the established systems for sparse linear algebra over non-floating point fields today are ill-suited to this task, due partly to a conflict of data structures;  standard software libraries generally use CSR, CSM, and triplet format, while high-performance persistence solvers rely on specialized ``lazy'' structures, as we will discuss in later sections.

The initiative to bring algebraic topology to data science is largely an initiative to bring \underline{\smash{linear homological algebra}} to data science.  Moreover, essentially all branches of computer linear algebra are about bases: how to compute them, how to work with them, etc.  To bring topological data analysis to fruition, therefore, the scientific community must   \underline{both} overcome the computational asymmetry described above, \underline{and} develop new algorithms for sparse linear algebra to  take advantage of modern, lazy data structures.

    \medskip
     
     \subsection{Lazy global duality}
Lazy methods have become key to modern methods of working with persistent Jordan bases, since $\EE$ has a tremendous number of nonzero entries, in general.
The archetypal example of a lazy method in this field is the computational technique introduced in \cite{bauer2019ripser}, which does not store the rows of a reduced matrix in memory, but rather re-constructs them each time they are needed. 
    At the cost of extra computation, such lazy methods circumvent some of the most challenging problems in computation for TDA, such as sparse matrices that are too large to store in memory.
    
    At present, most lazy methods focus on column access to  $(\EE^{-1})^\perp$.    
    The present work draws motivation, in part, from the problem of developing 
    lazy evaluation methods for both rows and columns of $\EE$ and $\EE^{-1}$.  
    
    \medskip
    
    \section{Contributions}
    
The primary contribution of this paper is a suite of \emph{lazy global duality} methods for computational TDA.
These methods grow from a novel,   memory-efficient sparse matrix factorization scheme, called \emph{U-match decomposition}. Once computed, a U-match decomposition of $\DD$ allows the user to retrieve persistent cycle \underline{and} cocycle representatives in a lazy fashion, using only a small number of algebraic operations.  

The  U-match decomposition algorithm that we present here performs only one  matrix reduction, and stores only one upper-triangular array in memory; by contrast, most state of the art primal-dual algorithms to compute PH generators currently  use two reductions and save two upper-triangular arrays.  In experiments, therefore, we  find that U-match decomposition can approximately halve time and memory costs for  generator computations, even after accounting for some dramatic performance improvements recently reported in \cite{vcufar2020ripserer}.  Moreover, the upper triangular array used in this decomposition happens to be computed as a byproduct by most state of the art persistent cohomology solvers; thus, the  approach is suitable for broad-scale implementation.

The U-match factorization scheme also permits lazy solution to a range of highly versatile problems in persistent homology, e.g. determining the filtration level at which two cycles become homologous.  These problems demonstrate, concretely, that the U-match scheme permits efficient access to information about a filtered complex which cannot be recovered from PH generators alone.

The  scheme has been implemented in the forthcoming software library \emph{ExHACT} \cite{ExHACT}.  This implementation has several important properties. Like the algorithm deployed in \emph{Eirene} \cite{henselman2016matroid},  and another which was recently (independently) proposed in \cite{chacholski2020algorithmic}, our method performs only row reductions on the boundary matrices in each dimension. However, unlike \cite{henselman2016matroid} and \cite{chacholski2020algorithmic}, the new approach is ``left-looking''\footnote{Within  the context of sparse matrix algorithms, the term \emph{left-looking} refers roughly to processes that eliminate nonzero entries of a column $c$ only when $c$ is selected; by contrast, \emph{right-looking} algorithms eliminate entries in $c$ whenever certain other columns appear in a sequence.  Left-looking algorithms are more compatible with compressed sparse row and column data structures, and are therefore preferred in many applications.}
; this yields substantial advantages in terms of sparse matrix manipulation and storage. 

The implementation incorporates a number of performance-enhancing techniques, some of which  apply to  U-match decomposition in general.  These include  compression methods that discard  data associated to non-pivot rows and  short-circuiting techniques that both save algebraic operations and sparsify output arrays.  The implementation also deploys several methods specific to persistence computation, e.g. the so-called ``twist" \cite{CKPersistent11} or ``clear-and-compress" technique \cite{BKRClear14, ZCComputing05}. The implementation also dovetails with   matrix compression schemes that have  proved highly successful in reducing time and memory cost of computation \cite{bauer2019ripser}.

Taken together, these methods represent an  important step toward realizing the promise of homological algebra as a general analytic tool for large-scale data -- especially in applications that focus on a small subset of features (e.g.\ the longest bar), and in topological models that demand a wider range of algebraic machinery (e.g. cellular sheaves). 
\section{Organization}

The paper is organized as follows. In \S\ref{sec:literature}, we review existing literature on PH computation, and in \S\ref{sec:notation} we recall some terminology and set notation. In \S\ref{sec:executivesummary}, we introduce the U-match decomposition of a matrix and discuss its basic properties and applications. In \S\ref{sec:overallblockstructure} we describe a set of \emph{inner identities} which play a fundamental role in simplifying the computation and application of U-match decomposition. In \S\ref{sec:lazyumatch}, we discuss lazy methods for U-match decomposition. In particular,  we develop a simple, flexible, and efficient scheme to store a minimal collection of compressed sparse upper triangular matrices, and detail how to use them to efficiently recover entries of other matrices of interest in persistence computations. Each of the saved matrices has only as many rows and columns as pivot elements in the decomposition -- potentially many fewer than $\DD$.  In \S\ref{sec:gloabl_duality_with_umatch}, we connect U-match decomposition directly back to persistent homology computation via global duality as characterized in Theorem \ref{thm:globalduality}, and  discuss lazy methods for global duality.  In \S\ref{sec:factorization_algorithms}, we give algorithms for computing U-match and discuss methods for optimization. Finally, in \S\ref{sec:experiments}, we present the results of some numerical experiments with our implementation, and in \S\ref{sec:conclusion} we summarize and discuss future directions.

Appendix \ref{sec:blockidentities} provides an expanded discussion of certain technical block matrix identities, and Appendix \ref{sec:asideonordertheory} briefly sketches an interesting connection with order theory.  Appendix \ref{sec:earlystopping} details a performance-enhancing technique for sparsificaiton and early stopping during U-match factorization generally, and Appendix \ref{sec:lazy_jordan_alt} describes how this method can be applied to homology computations, in particular.

\section{Related literature}
\label{sec:literature}

Our approach bears close connection to several standard methods in sparse linear algebra, such as LU decomposition. However it differs in several important respects:
    
\begin{enumerate}
    \item Order of rows and columns is centrally important to persistent homology computations, and must be preserved.  Thus, much of the main-stream machinery for optimal pivot ordering (e.g.\ in LU decomposition) does not apply.
    \item Exact numerical precision is required for homological rank calculations, so methods to control numerical error are unnecessary.
    \item Unlike  classical sparse matrix data structures such as CSC, CSR, triplet, etc.,  sparse boundary matrices often admit  methods to generate both rows \underline{and} columns efficiently.  Moreover, it is often more natural to index these arrays by names of cells in a cell complex than to index by integers.  Our   scheme is well suited to accommodate and leverage these unique properties.
\end{enumerate}

The scheme also bears close connection to existing factorization methods in persistent homology, such as $R = DV$.  We discuss this connection, and introduce a new, detailed treatment of several  matrix identities which are practically significant but have not been fully expanded in the current literature.

We will focus primarily on sequential approaches to persistent homology computation. Other, non-sequential approaches include the chunk algorithm \cite{BKRClear14},  spectral sequence procedures \cite{LSVspectral11, EHComputational10}, Morse-theoretic batch reduction \cite{HMM+Discrete14, HMM+Efficiency10, RWSTheory11, BLRandom13, GRW+Memoryar, henselman2016matroid, maria2019discrete, scaramuccia2020computing, DWComputing12}, distributed algorithms \cite{bauer2014distributed, morozov2020towards, lewis2015parallel},  GPU acceleration \cite{zhang2020gpu, hylton2019tuning}, streaming \cite{kerber2019barcodes}, and homotopy collapse \cite{botnan2015approximating, dey2019simba, boissonnat2018strong}.
 There are closely related techniques in matrix factorization and zigzag persistence \cite{milosavljevic2011zigzag, carlsson2019persistent, CSZigzag10}.
 
 The first algorithm to compute persistent homology was introduced and subsequently refined/expanded in  \cite{ELZTopological02, ZCComputing05, cohen2006vines}.  It is commonly known as the \emph{standard algorithm}, and constitutes a central pillar of persistent homology computation today, both in theory and applications.  
 It has become convention to express the output of the standard algorithm in terms of a so-called $R = DV$  decomposition described in the introduction.
 The underlying algebra concerning Smith normal form was expanded in \cite{skraba2013persistence}. 
 A corresponding treatment for persistent (relative) cohomology was developed in \cite{ MSVPersistent11, de2011dualities}, applying the standard algorithm to the anti-transpose of the differential matrix. There are a wide range of implementations, including \cite{
            bauer2019ripser, 
            bauer2014distributed, 
            bauer2017phat,
            vcufar2020ripserer,
            dey2014simpers,
            dey2019simba,
            fasyrtda,
            kaji2020cubical,
            lesnickrivet,
            maria2014gudhi,
            morozovdionysus,
            nandaperseus,
            deygicomplex,
            perryplex,
            TVAjavaPlex12,
            zhang2020gpu,
            zhang2019hypha}.    
 
 The cohomology algorithm is significantly faster for filtered clique complexes, empirically.
 This phenomenon been has widely replicated \cite{de2011dualities, otter2017roadmap}, and substantial work has been devoted to understanding this asymmetry; see for example \cite{zhang2019hypha, bauer2019ripser}.
 
 The barcodes produced by persistent homology and persistent (relative) cohomology are the same  \cite{de2011dualities}; a majority of existing state of the art solvers therefore employ variants of the cohomology algorithm.
 Existing applied and theoretical work seeks to leverage this asymmetry to accelerate computation of generators in homology.
 The Eirene library accomplishes this via a right-looking block reduction method which iteratively computes Schur complements and is therefore (in principle) agnostic to row versus column operations.  The Ripserer library accomplishes this by first identifying pivot elements via elementary row operations, thereby reducing the scope of column operations to pivot columns \cite{vcufar2020ripserer}; this offers impressive computational advantages, as non-pivot columns account for a disproportionate number of algebraic operations.
    Recently, a purely row-based algorithm to compute generators in PH, developed via categorical homotopy theory,   has been proposed in \cite{chacholski2020algorithmic}.
 These methods relate, in a loose conceptual sense, to several other acceleration techniques which leverage knowledge about order, sparsity, and linear dependence to reduce the number of algebraic operations needed to perform persistence computations \cite{ZCComputing05, BKRClear14, CKPersistent11, lampret2020chain, bauer2019ripser,  henselman2016matroid}.

The lazy regime developed in this work draws from and overlaps with these existing techniques to a great extent.  
It naturally incorporates the clear/compress/twist optimization and the short-circuiting technique associated with certain Morse vector fields, as  described in \S\ref{sec:shortcircuitph}.  The decomposition procedures presented in Algorithms \ref{alg_lrdec} and \ref{alg:revised_lrdec} are nearly identical, mathematically, to the cohomology algorithm, which applies the standard algorithm to the antitranspose of $\DD$; the primary functional difference lies only with the type of  information that is stored versus discarded on each iteration.  Moreover, the use of lazy data structures in these algorithms follows from the pioneering example of \cite{bauer2019ripser}, which was also a source of broader inspiration for this work.  Finally, many of the concepts introduced in our discussion have immediate analogs for $R = DV$ decomposition, which were largely worked out in \cite{de2011dualities}.

Nevertheless, our approach departs from existing techniques in the following particulars:

\begin{enumerate}
    
    \item While the use of row operations to compute PH generators has been developed explicitly via the machinery of matroids \cite{henselman2016matroid} and  homotopy \cite{chacholski2020algorithmic}, and implicitly via global duality \cite{de2011dualities}, the \underline{\smash{synthesis of global duality with lazy techniques}} is novel, to the best of our knowledge. 
    
    \item Unlike the right-looking methods implemented in Eirene \cite{henselman2016matroid} and proposed in \cite{chacholski2020algorithmic}, the U-match decomposition procedures described in this work (Algorithms \ref{alg_lrdec} and \ref{alg:revised_lrdec}) look left.  This has substantial significance for sparse matrix manipulation and storage.
    
    \item U-match decomposition itself presents an elegant framework to study filtered chain complexes which extends strictly beyond homological persistence.  A collection of illustrative examples appears in \S\ref{sec:lazy_global_duality}.

\end{enumerate}

\section{Notation and conventions}
\label{sec:notation}

In this section, we define notation and conventions used throughout this work. Write $\upto{m}$ for the ordered sequence $(1, \dots, m)$.  Throughout the text, $\field$ denotes a field.

Given matrix $\DD \in \field^{m \times n}$ and sequences $\sigma = (s_1, \ldots, s_p)$, $\tau = (t_1, \ldots, t_q)$ with $p \le m$, $q \le n$, we write $\DD_{\sigma \tau}$ for the matrix such that $\DD_{\sigma \tau}[i,j] = \DD[s_i, t_j]$.  That is,
\begin{align*}
    \DD_{\sigma \tau}
    \quad 
    =
    \quad
    \left [
    \begin{array}{ccc}
          \DD[s_1, t_1] & \cdots & \DD[s_1, t_q]   \\   
         \vdots & & \vdots \\
      \DD[s_q, t_1]    & \cdots &   \DD[s_p, t_q] \\ 
    \end{array}    
    \right ]
\end{align*}

In several cases, we wish to express that two matrices are equal up to a certain permutation of rows and columns.  In such cases we  use the symbol $\equiv$, together with row and column labels that indicate the permutation.  Thus, for example, we may write
\begin{align*}
    \begin{array}{l |cc|}
        \multicolumn{1}{c}{}& \multicolumn{1}{c}{} & \multicolumn{1}{c}{}  \\  \cline{2-3}
         & a & b   \\   
         & c & d \\ \cline{2-3}     
    \end{array}  
    \quad
    \equiv
    \quad
    \begin{array}{l |cc|}
        \multicolumn{1}{c}{}& \multicolumn{1}{c}{2} & \multicolumn{1}{c}{1}  \\  \cline{2-3}
         1 & b & a   \\   
         2 & d & c \\ \cline{2-3}     
    \end{array}      
    \quad
    \equiv
    \quad
    \begin{array}{l |cc|}
        \multicolumn{1}{c}{}& \multicolumn{1}{c}{1} & \multicolumn{1}{c}{2}  \\  \cline{2-3}
         2 & c & d   \\   
         1 & a & b \\ \cline{2-3}     
    \end{array}      
\end{align*}
If we omit labels on the rows or columns of a matrix $\DD \in \field^{m \times n}$, the implicit ordering is that of $\upto{m}$ and $\upto{n}.$

By a \emph{matching} between $S$ and $T$ we mean a partial matching on the directed bipartite graph with vertex set $(S,T)$.  In concrete terms, this means a subset $\emptyset\neq\IMatch \subseteq S \times T$ such that $(s,t) = (s', t')$ whenever  $s = s'$ or $t = t'$  for some   $(s,t) , (s', t') \in \IMatch$.

\section{U-match decomposition}
\label{sec:executivesummary}

A \emph{U-match decomposition} of a matrix $\DD \in \field^{m \times n}$ is a tuple $(\RR, \MM, \DD, \CC)$ such that 
    \begin{align}
    \RR \MM = \DD \CC
    \label{eq:umatchdef}
    \end{align}
where $\MM$ is a matching matrix and $\RR$ and $\CC$ are each upper unitriangular.  The choice of $\RR$ and $\CC$ as symbols is arbitrary, but it provides a useful mnemonic, since $\RR$ has an equal number of rows to $\DD$, and $\CC$ has an equal number of columns to $\DD$.

Throughout the discussion, we will identify $\DD$ with the linear map $\field^n \to \field^m$ given by left-multiplication with $\DD$.  Under this convention, the equation that defines a U-match decomposition corresponds to a commutative diagram 
\begin{equation}
\begin{tikzcd}
\field^n\arrow[r, "\MM"]\arrow[d, "\CC"] &\field^m\arrow[d, "\RR"]\\
\field^n\arrow[r, "\DD"]& \field^m
\end{tikzcd}
\label{eq:umatch}
\end{equation}

U-match decomposition is rich in mathematical structure, but one property above all will lead our discussion.  This is the \emph{matching} property, and it can be deduced directly from Equation \eqref{eq:umatchdef}.  
The matching property asserts that left-multiplication with $\DD$ does one of two things to each column $c$ of matrix $\CC$: (i) send $c$ to 0 -- in this case we call $c$  \emph{unmatched} -- or (ii) or send $c$ to a nonzero scalar multiple of some column $b$ of $\RR$ -- in this case $b$ and $c$ are \emph{matched}.  See Example \ref{ex:umatch} for illustration.

In fact, this example illustrates something more precise: not only does $\DD$ map columns of $\CC$ to columns of $\RR$, but $\MM$ completely determines the mapping.  More precisely,\footnote{Equation \ref{eq:matchingproperty_proto} is an equivalent form of the \emph{matching identity} defined in \S\ref{sec:matchingidentity}.} 
    \begin{align}
    \MM[i,j] \neq 0 \implies \DD \cdot \col_j( \CC) =\MM[i,j] \cdot  \col_i(\RR) 
    \label{eq:matchingproperty_proto}
    \end{align}
Thus matrix  $\MM$ tells us which column of $\CC$ maps to which column of $\RR$ and with what scaling factor.  Moreover, because $\MM$ is a matching matrix, we can infer that no two columns of $\CC$ map to the same column of $\RR$.  Thus, left-multiplication by $\DD$ determines a \emph{bona-fide}  matching between columns of $\CC$ and (nonzero scalar multiples of) columns of $\RR$.  Combining this with the fact that $\RR$ and $\CC$ are upper triangular, we arrive at the name \emph{U-match}.

Every part of the U-match decomposition has a name.   We call $\DD$ the \emph{mapping array} and $\MM$ the \emph{matching array}, respectively.  Since the columns of $\CC$ and $\RR$ form ordered bases, we refer to these as \emph{columnar ordered matching bases}, or \COB{}s.  To distinguish the two, we call $\CC$ the \emph{domain} \COB{} and $\RR$ the \emph{codomain} \COB{}.

\begin{example}  
\label{ex:umatch}
The following is a U-match decomposition.
    \begin{align}
    \underbrace
        {
            \begin{array}{|cc|}
                \cline{1-2}
                 1 & 1   \\   
                 &  1 \\ \cline{1-2}     
            \end{array}     
        }
        _
        { \scriptsize\Centerstack{\text{$\RR$} \\ \text{codomain} \\ \text{\COB{}}}}
    \quad   
    \underbrace
        {
            \begin{array}{|cc|}
                \cline{1-2}
                &    \\   
                3 &   \\ \cline{1-2}     
            \end{array}             
        }
        _
        { \scriptsize\Centerstack{\text{$\MM$} \\\text{matching} \\ \text{array}} }    
    \quad \quad
    =
    \quad \quad
    \underbrace
        {
            \begin{array}{|cc|}
                \cline{1-2}
                 3 &  -6   \\   
                 3 &  -6 \\ \cline{1-2}     
            \end{array}     
        }
        _
        { \scriptsize\Centerstack{\text{$\DD$} \\\text{factored} \\ \text{array}} }
    \quad    
    \underbrace
        {
            \begin{array}{|cc|}
                \cline{1-2}
                1 & 2   \\   
                 & 1   \\ \cline{1-2}     
            \end{array}             
        }
        _
        { \scriptsize\Centerstack{\text{$\CC$} \\\text{domain} \\ \text{\COB{}}}}   
        \label{eq:umatchexample}
	\end{align}  
Left-multiplication with $\DD$ maps the first column of $\CC$ to a scalar multiple of the second column of $\RR$:  
    \begin{align}
    \underbrace
        {
            \begin{array}{|c|}
                \cline{1-1}
                 1   \\   
                 1 \\ \cline{1-1}     
            \end{array}     
        }
        _
        { \col_{2}(\RR) }
    \;\;
    \cdot
    \;\;
    \underbrace
        {
            3    
        }
        _
        { \MM[2,1]  }    
    \quad \quad
    =
    \quad \quad
    \underbrace
        {
            \begin{array}{|cc|}
                \cline{1-2}
                 3 &  -6   \\   
                 3 &  -6 \\ \cline{1-2}     
            \end{array}     
        }
        _
        { \DD }
    \;\;
    \cdot
    \;\;
    \underbrace
        {
            \begin{array}{|c|}
                \cline{1-1}
                 1   \\   
                 0 \\ \cline{1-1}     
            \end{array}     
        }
        _
        { \col_{1}(\CC) }      
	\end{align}  
We say that these two columns are matched.  
Multiplication with $\DD$ sends the second column, $\col_2(\CC) = \bigl[\!\begin{smallmatrix} 2 \\ 1 \end{smallmatrix}\!\bigr]$, to 0, so we say that this column is unmatched.
\end{example}

We will reformulate the matching property in terms of the so-called  \emph{matching identity} in \S\ref{sec:matchingidentity} (Lemma \ref{lem:matchingidentity}).  This will provide a useful closed-form expression for $\DD \cdot \col_c(\CC)$  in terms of $\CC$ and $\MM$.  To do so, however, we must first introduce some notation for indexing.

\subsection{Row and column operation matrices}

Multiplying any U-match decomposition $\RR \MM = \DD \CC$ on the left with $\Ri$ yeids an equivalent identity, $\MM = \Ri \DD \CC$.  Viewed in this light, we can regard $\MM$ as the result of performing a sequence of column-operations (adding columns left-to-right) and row operations (adding rows bottom\footnote{Here ``bottom'' means bottom of the page; thus, for example, we might add a multiple of $\row_2(\DD)$ to $\row_1(\DD)$, but not vice versa.} to top) on $\DD$.  For this reason, we refer to $\CC$ as the \emph{column operation matrix} of the decomposition, and $\Ri$ as the \emph{row operation matrix}.  Row and column operation matrices are \underline{\smash{not uniquely defined}}, c.f. \S\ref{sec:uniqueCOMB}.

\subsection{Indexing}
\label{sec:indexing}

Since U-match decomposition is largely about matching, we  need precise notation to talk about matched indices.

 The \emph{index matching} of a U-match decomposition $\RR\MM = \DD \CC$ is the support of $\MM$.  Concretely, this means the set relation  $\IMatch \su \{1, \dots, m\} \times \{1, \ldots, n\}$ between the  row and column indices of $\MM$ such that 
    \begin{align*}
        (r,c) \in \IMatch        
        \iff
        \MM[r,c] \neq 0 .
    \end{align*}
The elements $(x,y)$ of this set correspond exactly to the \underline{\smash{pivot indices}} of the Gaussian elimination procedure (Algorithm \ref{alg_lrdec}) described in \S\ref{sec:factorization_algorithms}.  Indeed, the language of pivots and pivot elements permeates the literature on this style of decomposition.   We will therefore use the terms ``pivot index'' and ``matching index'' interchangeably.

The  \emph{definition} and \emph{values} of $\IMatch$ are $\DEF(\IMatch) = \{ p :  (p, q) \in \IMatch \}$ and $\VAL(\IMatch) = \{q : (p,q) \in \IMatch\}$, respectively.  There are natural bijections 

\begin{equation}
\begin{tikzcd}
{\DEF(\IMatch)} \arrow[r, shift left= 0.5ex, "\mCol"]
&
{\VAL(\IMatch)}\arrow[l, shift left= 0.5ex, "\mRow"]
\end{tikzcd}
\end{equation}
such that $\mRow(c) = r$ and $\mCol(r) = c$ for each $(r,c) \in \IMatch$

The row indices contained  in  $\DEF(\IMatch)$ and in the  set  complement  $ \{1, \ldots, m\} \backslash \DEF(\IMatch)$ can be arranged into strictly increasing sequences  $\prows = (\prow_1, \cdots,  \prow_k)$ and $\nprows = (\nprows_1, \ldots, \nprows_{n-k})$, respectively.  
A similar convention applies to  $\VAL(\IMatch)$ and $ \{1, \ldots, n\} \backslash \VAL(\IMatch)$, as expressed in the following table.

\begin{center}
    \begin{tabular}{rll}
  \toprule
        &rows&columns\\
        \midrule
    all indices&$\upto{m} = (1 < \cdots < m) $
    &
    $\upto{n} = (1< \cdots < n)$\\
    
        \midrule
    matched indices&$\prows  = ( \prow_1 < \cdots < \prow_k )$
    &
    $\pcols  =  ( \pcol_1 < \cdots < \pcol_k )$\\
    \midrule
    unmatched indices&$    \nprows  = ( \nprow_1 < \cdots < \nprow_{m-k} )$
    &
    $\npcols  =  ( \npcol_1 < \cdots < \npcol_{n-k} )$    \\
    \bottomrule
\end{tabular}
\end{center}
We further write $\prowp_p = \mCol(\prow_p)$ for the  column matched to  row $\prow_p$, and $\pcolp_p = \mRow(\pcol_p)$ for the  row matched to column $\pcol_p$.  Thus 
    \begin{align*}
        \IMatch
        & \;= \;
        \{ (\prow_1, \prowp_1), \ldots, (\prow_k, \prowp_k) \} \\
        &
        \; = \; 
        \{ (\pcolp_1, \pcol_1), \ldots, (\pcolp_k, \pcol_k) \} 
    \end{align*}
Note, however, that 
    \begin{align*}
        (\prow_p, \prowp_p) \neq (\pcolp_p, \pcol_p)
    \end{align*}
in general.

\subsection{The matching identity}
\label{sec:matchingidentity}

We may now formalize the statement that left-multiplication with $\DD$ ``matches columns to columns.''  Equation \eqref{eq_matchcolumn} contains both the statement and the proof.  Equation \eqref{eq_matchrow} encapsulates the dual statement that right-multiplication with $\DD$ ``matches rows to rows.''


\begin{lemma}[Matching identity for columns]
\label{lem:matchingidentity}
For any U-match decomposition $\RR \MM = \DD \CC$, one has 
    \begin{align}
        \DD \cdot \col_c(\CC) 
        & = 
        \col_c(\DD \cdot \CC) 
        \nonumber \\
        &=
        \col_c(\RR \cdot \MM)
        \nonumber\\
        &=
            \begin{cases}
                \col_{\mRow(c)} (\RR) \cdot \MM[\mRow(c), c] & c \text{ is a matched index } \\
                0 & else
            \end{cases}
        \label{eq_matchcolumn}
    \end{align}
\end{lemma}

\begin{lemma}[Matching identity for rows]
\label{lem:matchingidentityDUAL}
For any U-match decomposition $\RR \MM = \DD \CC$, one has 
    \begin{align}
        \row_r( \Ri) \cdot \DD 
        & = 
        \row_r( \Ri \cdot \DD )
        \nonumber \\
        &=
        \row_r(\MM \cdot \Ci)
        \nonumber\\
        &=
            \begin{cases}
                \row_{\mCol(r)} (\Ci) \cdot \MM[r, \mCol(r)] & r \text{ is a matched index } \\
                0 & else
            \end{cases}
        \label{eq_matchrow}
    \end{align}
\end{lemma}

\subsection{Linear dual spaces and matching}

The matching identity for rows (Lemma \ref{lem:matchingidentityDUAL}) has a natural interpretation in terms of linear dual spaces; this interpretation is not vital to our main narrative, but it does lend a complementary perspective.

Each length-$n$ row vector $v^*$ naturally determines a linear map $\field^n \to \field, \; u \mapsto v^* \cdot u$ from the space $\field^n$ of length-$n$ column vectors to the ground field $\field$.  Since $\Ci \CC = I$, the set $\{ \row_i(\Ci) : i \le n \}$ forms a basis dual to $\{ \col_j(\CC) : j \le n\}$, in the sense that $\row_i(\Ci) \cdot \col_j(\CC) = \delta_{ij}$, the Kronecker delta.

Viewed in this light, the matching identity for rows  states that the linear map of dual spaces $(\field^m)^* \to (\field^n)^*, \; v^* \mapsto (v^* \circ \DD)$ carries each element of the basis dual to $\RR$ either to 0 or to a nonzero scalar multiple of some element in the basis dual to $\CC$.

\subsection{Proper decomposition}

A  {U-match decomposition} $\RR \MM  = \DD \CC$ is \emph{proper} if both of the following axioms hold, where $\id^{n \times n}$ denotes the $n \times n$ identity matrix: 
    \begin{enumerate}[label=\textbf{(A\arabic*)}]
        \item \label{item:proper1} $\col_k(\MM) = 0 \implies \row_k(\CC)$ is the $k$th standard unit row vector.  This is equivalent to the condition that
            \begin{align*}
                \CC_{\upto{n} \npcols} = \id^{n\times n}_{\upto{n} \npcols}.
            \end{align*}
        \item \label{item:proper2} $\row_k(\MM) = 0 \implies \col_k(\RR)$ is the $k$th standard unit column vector; equivalently, 
            \begin{align*}
                \RR_{ \nprows \upto{m}} = \id^{m\times m}_{\nprows \upto{m}}.
            \end{align*}
    \end{enumerate}
    
Proper decompositions have useful properties we that will leverage in later sections.  Most of the algorithms that compute U-match decomposition actually compute \emph{proper} decompositions.

\subsection{Existence}

Every matrix $\DD$ admits a proper U-match decomposition; \emph{a fortiori}, every matrix has a U-match decomposition.  We provide a constructive proof in \S\ref{sec:factorization_algorithms} (Algorithm \ref{alg_lrdec} and Proposition \ref{prop:correctnessofalgorithm1}).
These results are technical in nature, so we defer them to the end of the discussion.

\subsection{Uniqueness of matching arrays, non-uniquness of \COB{}s, and codetermination of proper \COB{}s}
\label{sec:uniqueCOMB}

A single matrix $\DD$ may admit multiple distinct U-match factorizations, and indeed multiple distinct proper U-match factorizations.  However, the associated matching array is unique.

\begin{example}[Proper \COB{}s are not uniqu]
Since $N \id = \id N$, any upper-unitriangular matrix can be a (co)domain \COB{} for mapping array $\DD = \id$.  
\end{example}

\begin{theorem}[Matching arrays are unique]
\label{thm:matchingunique}
Let $\DD$ be a matrix. For any two U-match decompositions $\RR \MM = \DD \CC$ and $\stilde \RR \stilde \MM = \DD \stilde \CC$, one has $\MM = \stilde \MM$.
\end{theorem}
\begin{proof}
By vertically concatenating the associated diagrams from  Equation \eqref{eq:umatch} and reversing some arrows, one can obtain a U-match decomposition of form $\stilde{\RR}^{-1} \RR \MM = \stilde \MM \stilde{\CC}^{-1} \CC$. Multiplication by invertible upper triangular matrices does not change the rank of any submatrix in the lower-lefthand corner of an array, so the number of nonzero entries in each $p \times q$ submatrix in the lower lefthand corners of $\MM$ and $\stilde \MM$ is the same.  Thus, $\MM$ and $\stilde \MM$ have equal sparsity patterns.  Furthermore, every nonzero entry of a matching array remains the same after multiplication by a unitriangular matrix, so the nonzero elements of $\MM$ and $\stilde \MM$ are identical.
\end{proof}

In light of Theorem \ref{thm:matchingunique}, we may speak  of the  \underline{\smash{unique matching array}} associated to $\DD$ by U-match decomposition.\footnote{Theorem \ref{thm:matchingunique} reflects a deeper structural result concerning bifiltrations of linear spaces (or, more generally, modular lattices), c.f.\  Appendix \S\ref{sec:asideonordertheory}.}

However, in proper U-match decompositions $\DD$ determines three out of four blocks of both $\RR$ and $\CC$.  This can be shown using the ``inner identities'' which will be proved in Theorem  \ref{thm:umatchblockidentities}.

\begin{proposition}[Row and column operation matrices are ``mostly'' unique]
\label{prop:1blockdiff}
If $(\RR, \MM, \DD, \CC)$ is a proper U-match decompositions, then 
\begin{align*}
    \Ri
    \equiv
    \begin{array}{l |cc|}
        \multicolumn{1}{c}{}& \multicolumn{1}{c}{\nprows} & \multicolumn{1}{c}{\prows}  \\  \cline{2-3}
        \nprows  & \id & - \DD_{\nprows \pcols} \Drki   \\   
        \prows   & 0 & * \\ \cline{2-3}     
    \end{array}      
    &&
    \CC
    \equiv
    \begin{array}{l |cc|}
        \multicolumn{1}{c}{}& \multicolumn{1}{c}{\pcols} & \multicolumn{1}{c}{\npcols}  \\  \cline{2-3}
        \pcols  & * & -\Drki \DD_{\prows \npcols}   \\   
        \npcols   & 0 & \id \\ \cline{2-3}     
    \end{array}      
\end{align*}
Thus, $\DD$ uniquely determines three out of four blocks in both $\RR$ and $\CC$.
\end{proposition}
\begin{proof}

It follows from the inner identities (Theorem \ref{thm:umatchblockidentities}) that 
    $
        \RR_{\nprows\prows} = -\DD_{\nprows  \pcols} (\DD_{\prows  \pcols})^{-1}
    $
    and
    $
    \CC_{\pcols\npcols} = - (\DD_{\prows  \pcols})^{-1} \DD_{\prows  \npcols}
    $.  The fact that $\Ri_{\upto{m} \nprows} = \id^{m \times m}_{\nprows}$ and $\CC_{\upto{n} \npcols} = \id^{n\times n}_{\upto{n} \npcols}$ follows from axioms \ref{item:proper1} and \ref{item:proper2}.    The desired conclusion follows.
\end{proof}

 Moreover, a proper domain \COB{} uniquely determines the corresponding codomain \COB{}, and vice versa:

\begin{proposition}[Proper \COB{}s codetermine]
\label{prop:LRcorrespondence} If $\RR \MM = \DD \CC$ and $\stilde \RR \MM = \DD \stilde \CC$ are proper U-match decompositions, then $\RR = \stilde \RR$ if and only if  $\CC = \stilde \CC$.
\end{proposition}
\begin{proof}
Theorem \ref{thm:umatchblockidentities} shows that $\RR$ uniquely determines  $\CC$.  The converse holds by the anti-transpose symmetry, which is defined in \S\ref{sec:antitransposesymmetry}.
\end{proof}

\begin{remark}
If $\DD$ is an invertible square matrix then every row index matches to a column index and vice versa.  Thus, for invertible $\DD$, every U-match decomposition is proper.  Proposition \ref{prop:LRcorrespondence} implies that in this case $\RR$ uniquely determines $\CC$ and vice versa.  This observation serves as a sanity check, since the same conclusion can be deduced directly from ordinary matrix algebra.  
\end{remark}

\subsection{Anti-transpose symmetry}
\label{sec:antitransposesymmetry}

The \emph{anti-transpose} of a matrix $\DD$ is the matrix $\DD^\perp$ obtained by transposing $\DD$, then reversing the order of rows and columns.
To each (proper) U-match decomposition $\RR \MM = \DD \CC$ corresponds an ``anti-transposed'' (proper) U-match decomposition  $(\CC^{-1})^{\perp} \MM^{\perp} = \DD^\perp (\RR^{-1})^\perp$.  This transformation carries the domain \COB{} of the original decomposition to the codomain \COB{} in the new decomposition, and vice versa.

The significance of the anti-transpose for persistent (co)homology  was first explored in \cite{de2011dualities};  much of the current discussion grows  from that work.

\subsection{Subspace bases}
\label{sec:fundamentalsubspaces}

U-match decomposition provides an elegant means to solve one of the most versatile problems in computational algebra: constructing a basis for a subspace.  As we will see, U-match decomposition yields bases for images, kernels, and inverse images not only of $\DD$, but of every upper-lefthand block submatrix of $\DD$.

\subsubsection*{A lattice of subspaces}

Formally, let us define filtrations
\begin{alignat}{6}
    0 &= \filta_0 \field^n &&\subseteq \cdots &&\subseteq  \filta_n \field^n  &&= \field^n
    \label{eq:domainfiltration} \\
    0 &= \filtb_0 \field^m &&\subseteq \cdots &&\subseteq  \filtb_m \field^m  &&= \field^m
    \label{eq:codomainfiltration}
\end{alignat}
where $\filta_\fparam \field^n$ is the subspace of $\field^n$ consisting of vectors supported on $\upto{p}$, and $\filtb_\fparam \field^n$ is defined similarly as a subspace of $\field^m$.  By abuse of notation, we will sometimes abbreviate these expressions to $\filta_\fparam$ and $\filtb_\fparam$.  We will likewise write $\filta_*$ and $\filtb_*$ for the nested sequences  \eqref{eq:domainfiltration} and \eqref{eq:codomainfiltration}, respectively. Further, let us write
    $
    \DD_\bullet \filta_\fparam  = \{ \DD v : v \in \filta_\fparam \}
    $
    and 
    $
    \DD^\bullet \filta_\fparam  = \{ v : Dv \in \filta_\fparam \}        
    $
for the direct image and inverse image of $\filta_\fparam$ under $\DD$, respectively. Then we have nested sequences
\begin{alignat*}{6}
    \Ker(\DD) &= \DD^\bullet \filtb_0 &&\subseteq \cdots &&\subseteq  \DD^\bullet \filtb_m  &&= \field^m
    \\
    0 &= \DD_\bullet \filta_0 &&\subseteq \cdots &&\subseteq  \DD_\bullet \filta_n  &&= \Im(\DD) 
\end{alignat*}
denoted $\DD^\bullet \filtb_*$ and $\DD_\bullet \filta_*$, respectively.

We show (Theorem \ref{thm:fundamentalsubspaces}) that if $\RR \MM = \DD \CC$ is a U-match decomposition and $\cols(\CC)$ is the set of columns of $\CC$, then $\cols(\CC)$, contains a basis for each subspace in either  $\filta_*$ or $\DD^\bullet \filta_*$.    It therefore contains a basis for the sum and intersection of any two spaces in  $\filta_*$ and $\DD^\bullet \filta_*$, by the following elementary result from linear algebra: 

\begin{lemma}\label{lem:subspace_bases}
Let $V$ be a vector space. If a basis $T$ for $V$ contains bases for subspaces $V_1$ and $V_2$, then it contains bases for $V_1 \cap V_2$ and $V_1 + V_2$.
\end{lemma}

\begin{proof}
Let $T_1, T_2 \subseteq T$ be bases for $V_1$ and $V_2$ respectively. Then $T_1 \cup T_2$ spans $V_1 + V_2.$ Since $T_1 \cup T_2 \subseteq T$ is linearly independent, it is thus a basis.

On the other hand, since $V_1\cap V_2 \subseteq V_1, V_2$, we can uniquely represent every element $V_1\cap V_2$ using linear combinations of elements of some subset $T_1' \subseteq T_1$ and of some subset $T_2'\subseteq T_2$. However, since $T_1', T_2' \subseteq T$, and there is a unique representation of every element of $V$ as a linear combination of elements of $T$, these two sets must be the same, and thus a basis for $V_1 \cap V_2.$
\end{proof}

In fact, Lemma \ref{lem:subspace_bases} implies more; if we write 
    $
    \subspacelattice^n 
    $
for the order lattice of linear subspaces of $\field^n$, and 
    $$
    \subspacelattice^n_\mathrm{dom}
    \subseteq
    \subspacelattice^n 
    $$
for the \underline{bounded order sublattice}\footnote{In this context, a sublattice is bounded if it contains $0$ and $\field^n$.} of $\subspacelattice^n$ generated by $\filta_*$ and $\DD^\bullet \filta_*$, then $\cols(\CC)$ contains a basis for every element of $\subspacelattice^n_\mathrm{dom}$ (Theorem \ref{thm:bifiltrationbasis}).    Likewise, if 
    $$
    \subspacelattice^m_\mathrm{cod}
    \subseteq
    \subspacelattice^m 
    $$
denotes the sublattice of $\subspacelattice^m $ generated by $\filtb_*$ and $\DD_\bullet \filta_*$, then $\cols(\RR)$ contains a basis for every element of $\subspacelattice^m_\mathrm{cod}$ (Theorem \ref{thm:bifiltrationbasis}).

\subsubsection*{Kernel and image}

The connection between these subspaces and the persistence computation is foreshadowed by the following language.  Let us say that a vector $v \in \filta_\fparam \field^n \backslash \filta_{\fparam-1} \field^n$ is \emph{born} at time $\fparam$, and a vector $u \in \DD_\bullet \filta_\fparam \backslash \DD_\bullet \filta_{\fparam-1}$ is \emph{bounded} by time $p$.  Then
    \begin{enumerate}
        \item $\filta_\fparamb \cap \DD^\bullet 0$ is the set of \underline{kernel vectors} born by time $\fparamb$ 
        \item $\filtb_\fparamb \cap \DD_\bullet \filta_\fparam$ is the set of \underline{\smash{image  vectors}} born by time $\fparamb$ that are bounded by time  $\fparam$
    \end{enumerate}

\subsubsection*{Selection of basis elements (quick reference)}

A note to those concerned more with the punchline than  the technical details of proof:

Bases for $\DD_\bullet \filta_\fparam$ and $\DD^\bullet \filtb_\fparam$ are given as follows:
\begin{align}
    \DD^\bullet \filtb_\fparam: 
        &&
        \{ \col_c( \CC) : \col_c(\MM)  \in \filtb_\fparam \}
        \label{eq:basispullback}
    \\        
    \DD_\bullet \filta_\fparam: 
        &&
        \{ \col_{ r }( \RR) : (r,c) \in \IMatch, \; \;  c \le \fparam \} 
        \label{eq:basispushforward}
\end{align}
To obtain a basis for $\filta_\fparamb \cap \DD^\bullet \filtb_\fparam$, simply remove any elements of Equation \eqref{eq:basispullback} that are  supported outside of  $\upto{\fparamb}$.  To obtain a basis for $\filtb_\fparamb \cap \DD_\bullet \filta_\fparam$, likewise  remove any elements of Equation \eqref{eq:basispushforward} that are  supported outside of  $\upto{\fparamb}$.

\subsubsection*{Technical statements}

\begin{theorem}
\label{thm:fundamentalsubspaces}
Let $\RR\MM = \DD \CC$ be any U-match decomposition. 
\begin{enumerate}
    \item The intersection $\cols(\CC) \cap \DD^\bullet \filta_\fparam$ is a basis for  $\DD^\bullet \filta_\fparam$.
    \item The intersection $\cols(\RR) \cap \DD_\bullet \filta_\fparam$ is a basis for $\DD_\bullet \filta_\fparam$.    
\end{enumerate}
\end{theorem}
\begin{proof}
The commutative diagram that defines U-match decomposition induces the following pair of  diagrams:
\begin{equation*}
\begin{tikzcd}
    \filta_\fparam \arrow[r, "\MM_\bullet"] \arrow[d,  shift left=0.6ex, "\CC_\bullet"]  \arrow[d,  shift right=0.6ex, leftarrow, "\CC^\bullet" ']
    &
    \MM_\bullet \filta_\fparam  \arrow[d,  shift left=0.6ex, "\RR_\bullet"]  \arrow[d,  shift right=0.6ex, leftarrow, "\RR^\bullet" ']
\\
    \filta_\fparam \arrow[r, "\DD_\bullet"]
    & 
    \DD_\bullet \filta_\fparam
\end{tikzcd}
\quad
\quad
\begin{tikzcd}
    \MM^\bullet \filta_\fparam \arrow[r, leftarrow, "\MM^\bullet"] \arrow[d,  shift left=0.6ex, "\CC_\bullet"]  \arrow[d,  shift right=0.6ex, leftarrow, "\CC^\bullet" ']
    &
    \filta_\fparam  \arrow[d,  shift left=0.6ex, "\RR_\bullet"]  \arrow[d,  shift right=0.6ex, leftarrow, "\RR^\bullet" ']
\\
    \DD^\bullet \filta_\fparam \arrow[r, leftarrow, "\DD^\bullet"]
    & 
    \filta_\fparam
\end{tikzcd}
\end{equation*}
Since direct and inverse image of isomorphisms also preserve bases, the desired conclusion follows from the (straightforward to verify) special case where $\DD$ is a matching matrix.
\end{proof}

\begin{theorem}
\label{thm:bifiltrationbasis}
If $\RR\MM = \DD \CC$ is any U-match decomposition, then
    \begin{enumerate}
        \item The columns of $\CC$ contain a basis for every element of $\subspacelattice^n_\mathrm{dom}$.
        \item The columns of $\RR$ contain a basis for every element of $\subspacelattice^n_\mathrm{cod}$.        
    \end{enumerate}
\end{theorem}

\begin{proof}
  The columns of $\CC$ contain bases for each subspace $\filta_p$ and, by Theorem \ref{thm:fundamentalsubspaces}, each subspace $\DD^\bullet \filta_p$. Applying Lemma \ref{lem:subspace_bases}, the first claim follows.  The second claim is argued similarly.
\end{proof}

\subsection{Solving systems of linear equations}

As with LU decomposition, U-match decomposition can provide an effective means to solve systems of linear equations. Let $x$ be a vector, and define its support to be the set of indices of non-zero entries of $x$, $$\supp(x) = \{p_1 < p_2 \cdots < p_k \} = \{ p  \;:\; x[p] \neq 0\}$$ with total ordering inherited from the  indexing set for the vector's entries.

\begin{proposition}
\label{prop:solvingsystems}
Let $\RR \MM = \DD \CC$ be a U-match decomposition.  
\begin{enumerate}
    \item There exists a solution $x$ to equation $\DD x = b$ iff the column vector $\Ri b$ vanishes on non-pivot row indices, i.e. $(\Ri b)_{\nprows} = 0$.  In this case
    \begin{align}
        x : = \CC_{\allc \pcols} (\MM_{\prows \pcols})^{-1} (\Ri)_{\prows \allr } b
        \label{eq:primalsolution}
    \end{align}
satisfies the equation.  
    \item Dually, there exists a solution $x$ to equation $y \DD = c$ iff the row vector $c \CC_{\npcols}$ vanishes on non-pivot column indices, i.e. $ (c\CC)_{\npcols})=0$.  In this case 
    \begin{align}
        y : = c \CC_{\allc \pcols} (\MM_{ \prows \pcols })^{-1} (\Ri)_{\prows \allr} 
        \label{eq:dualsolution}        
    \end{align}    
satisfies the equation.    
\end{enumerate}
\end{proposition}

\begin{proof}
Claim 1 follows from the observation that $b = \DD x  \implies  \Ri b = \Ri \DD x =  \MM \Ci x $. Since the non-pivot rows of $\MM \Ci$ vanish, the non-pivot rows of $\Ri b$ must vanish, also.  Conversely, if we assume that these coefficients do in fact vanish, then the correctness of the proposed solution can be confirmed by direct substitution, if we recall the fact that $\DD = \RR \MM  \Ci$. The second claim is similar.
\end{proof}

\newcommand{\qhat}{\hat q }

\begin{corollary}  
\label{cor:minmaxsolutions}
Suppose that both problems stated in Proposition \ref{prop:solvingsystems} have solutions, and let $x$ and $y$ be the solutions defined in Equations \eqref{eq:primalsolution} and \eqref{eq:dualsolution}, respectively.
    \begin{enumerate}
    \item \emph{Minimality of primal solutions}.  Solution $x$ is minimal in the sense that $\max \supp(x) \le \max \supp (\hat x)$ for any $\hat x$ such that $D \hat x = b$.  
    \item \emph{Maximality of dual solutions}.  Solution $y$ is maximal in the sense that $\min \supp(y) \ge \min \supp( \hat y)$ for any $\hat y$ such that $\hat y D = c$.
    \end{enumerate}
\end{corollary}
\begin{proof}
Let $q = \max \supp(x)$.  By construction, $\MM[p,q] \neq 0$ for some $p$ such that $(\Ri b)[p] \neq 0$.  

Suppose, for a contradiction, that there exists a solution $\DD \hat x = b$ such that $\qhat = \max \supp(\hat x) < q$. 
Then $b$ lies in the column space of $\DD_{\upto{m} \upto{ \qhat}}$.  However, one can check that $\RR \MM_{\upto{m} \upto{ \qhat }} = \DD_{\upto{m}\upto{ \qhat }} \CC_{\upto{ \qhat }, \upto{ \qhat }}$ is a valid U-match decomposition, and the accompanying matching consists of all pairs $(r,c)$ such that $\MM[r,c] \neq 0$ and $c \le \qhat < q$.  These pairs do not include $(p,q)$ since $q$ is too big; however $(\Ri b)[p] \neq 0$ as we have already observed.  Thus, we derive a contradiction, via claim 1.  The corresponding statement for $y$ holds by anti-transpose symmetry. 
\end{proof}

In the special case of a linear equation $\CC x = b$, where $b \in \Ker(\DD)$, U-match decomposition provides an especially simple closed form solution. To state the solution in symbols, let us write
    \begin{align*}
        \del_T(v)[p]
        = 
        \begin{cases}
            0 & p \in T \\
            v[p] & else \\
        \end{cases}    
    \end{align*}
for any vector $v \in \field^m$ and any set $T \su \upto{m}$.    

\begin{proposition}[Solving for kernel vectors]
\label{prop:kerneldeletion}
Let $\RR \MM =  \DD \CC$ be a proper U-match decomposition.  Then
    \begin{align*}
        \DD b = 0 \quad &\implies \quad  \Ci b = \del_\pcols(b) \\
        c \DD = 0 \quad &\implies \quad  c \Ri = \del_{\prows}(c)
    \end{align*}
\end{proposition}
\begin{proof}
Fix $b$ such that $\DD b = 0$.  The set $\{\col_p(\CC) : p \in \npcols \}$ is a basis for $\Ker(\DD)$, by Theorem \ref{thm:fundamentalsubspaces}, so there exists a unique vector of coefficients $a$ such that $b = \sum_{p \in \npcols} a[p] \col_p(\CC)$.  Axiom \ref{item:proper1} of proper U-match decomposition implies that $\col_p(\CC)[q] = \delta_{pq}$, the Dirac delta, for any $p, q \in \npcols$.  Therefore $a[p]$ must equal $b[p]$ for all $p \in \npcols$.  This establishes the first implication.  The second follows by anti-transpose symmetry.
\end{proof}

\begin{remark}
A special case of Proposition \ref{prop:kerneldeletion} was noted in some of the earliest papers on persistent homology computation \cite{ZCComputing05}, where it was leveraged to reduce the number of algebraic operations performed.  The idea  was  adapted in \cite{CKPersistent11, BKRClear14}, which introduce computation with a ``twist,'' clearing, and compression.
\end{remark}

\subsection{Related factorizations}

U-match decomposition has several simple relationships with other factorization schemes.

\subsubsection*{LU decomposition}

\begin{lemma}[LU decomposition]
\label{lem:lu}
Each U-match decomposition $\RR \MM =\DD \CC$ corresponds to an LU decomposition  $LP =  N U$  where 
    \begin{align*}
    P = Q \MM_{\prows  \pcols} && L = Q\RR_{\prows  \prows}Q && N = Q\DD_{\prows  \pcols} && U = \CC_{\pcols\pcols}   
    \end{align*}
and $Q$ is an exchange matrix.    
\end{lemma}
\begin{proof}
Follows from the fact that $\RR_{\prows  \prows}\MM_{\prows  \pcols} =  \DD_{\prows  \pcols} \CC_{\pcols\pcols}$, since $\RR$ and $\CC$ are upper-triangular.
\end{proof}

\subsubsection*{Echelon form}

Proposition \ref{prop:echelonform} employs the matrix indexing notation described in \S\ref{sec:notation}.  The block structure of U-match decomposition is expanded at greater length in  \S\ref{sec:overallblockstructure}.

\begin{proposition}[Reduced echelon form]
\label{prop:echelonform}
Let 
\begin{align*}
    \Ri_{ech}
    :=
    \begin{array}{l |cc|}
        \multicolumn{1}{c}{}& \multicolumn{1}{c}{\nprows} & \multicolumn{1}{c}{\prows}  \\  \cline{2-3}
        \nprows & \id  &    \\   
        \prows &   & \MM_{\prows  \pcols}\AA^{-1}  \\ \cline{2-3}     
    \end{array}  
\end{align*}
Then the product $\Ri_{ech} \Ri$ is upper unitriangular, and up to (i) permutation of rows and columns, and (ii) scaling of leading row entries, the matrix product $\Ri_{ech} \Ri \DD$ is in reduced row echelon form. Likewise, if 
\begin{align*}
    \CC_{ech}
    :=
    \begin{array}{l |cc|}
        \multicolumn{1}{c}{}& \multicolumn{1}{c}{\pcols} & \multicolumn{1}{c}{\npcols}  \\  \cline{2-3}
        \pcols & (\MM_{\prows \pcols})^{-1} \Rirr \MM_{\prows  \pcols}  &    \\   
        \npcols &   & \id  \\ \cline{2-3}     
    \end{array}  
\end{align*}
Then the product $\CC\CC_{ech}$ is upper unitriangular, and up to (i) permutation of rows and columns, and (ii) scaling of leading row entries, the matrix product $\DD\CC\CC_{ech} $ is in reduced column echelon form.
\end{proposition}
\begin{proof}
Follows from the formulae for $\Ri \DD$ and $\DD \CC$ in Proposition \ref{thm:umatchblockidentities}.
\end{proof}

\subsubsection*{Right-reduction $(R = DV)$}

Right-reduction, or $R = DV$ decomposition as it is commonly known, was defined in \S\ref{subsec:computationalassymetry}.  We refer to coefficient  $R[ \low_R(p),p]$ as the pivot entry of matrix $R$ in column $p$; these so-called pivot entries are in fact the bona-fide pivot elements of a standardized Gauss-Jordan elimination process on $D$ \cite{cohen2006vines}.

We say that an $R=DV$ decomposition is \emph{unitriangular} if $V$ is unitriangular.  It is \emph{proper} if, in addition,  $V$ obeys the same  axiom as a domain matching basis in a proper U-match decomposition -- concretely, for each pivot element $(q, p)$, $\row_p(V)$ must be a unit vector.

\begin{lemma}[$R = DV$, general]
\label{lem:umatchRDVgeneral}
To each  $R = DV$ decomposition corresponds at least one $U$-match decomposition  $\RR \MM =  \DD V$, and vice versa.
\end{lemma}
\begin{proof}
Let $R = DV$  be given.  One can  transform $R$ into a matching array by elementary  operations which add lower rows to higher ones;  these operations yield a matrix equation $\MM = \Ri \DD V$ and left-multiplication by $\RR$ yields a U-match decomposition $\RR \MM = \DD V$. The converse follows from the observation that $\RR \MM$ is reduced for any U-match decomposition $\RR \MM =  \DD V$.
\end{proof}

\begin{lemma}[$R = DV$, proper]
\label{lem:RDVproperunique}
To each proper $R = DV$ decomposition corresponds a \emph{unique} proper $U$-match decomposition of form $\RR \MM =  \DD V$, and vice versa.
\end{lemma}
\begin{proof}
Existence follows from Lemma \ref{lem:umatchRDVgeneral}.  The codomain \COB{}, $\RR$, is uniquely determined by the domain \COB{}, $V$, by  Proposition \ref{prop:LRcorrespondence}, so the decomposition is unique. 
\end{proof}

\begin{remark} The standard algorithm for $R = DV$ decomposition \cite{cohen2006vines} therefore yields a simple method to obtain proper U-match decompositions.
\end{remark}

\section{Block submatrices}
\label{sec:overallblockstructure}

Block submatrices play an important role in the story of U-match decomposition.  In particular, the \emph{inner identities} described below  hold the key to defining and proving the correctness of the  lazy look-up procedures which will be described in \S\ref{sec:lazylookup}; these procedures, in turn, constitute one of the central contributions of the present work.  

\underline{Nota Bene:} the content of this section is primarily technical.   The reader may simply wish to skim the {inner identities} (\S \ref{sec:inneridentities}), and return when the need arises.

To begin, let us remind the reader of the notational conventions introduced in \S\ref{sec:notation} and \S\ref{sec:indexing}.  Under these conventions, $\DD$ has the following block structure: 
	\begin{align*}
	\DD
	\quad
	\equiv
	\quad
    \begin{array}{c |ccc|ccc|}
    \multicolumn{1}{c}{}& \multicolumn{1}{c}{\pcol_1} & \multicolumn{1}{c}{\cdots} & \multicolumn{1}{c}{\pcol_{\npiv}} &  \multicolumn{1}{c}{ \npcol_1 } & \multicolumn{1}{c}{\cdots} & \multicolumn{1}{c}{\npcol_{n-k}} \\  \cline{2-7}
     \nprow_1 &  &  &  & &  &  \\   
     \vdots &  & \DD_{\nprows \pcols } &  &  & \DD_{\nprows \npcols} &\\ 
     \nprow_{m-k} &  &  & &  & &\\ \cline{2-7}
     \prow_1 &  &  &  & & &\\ 
     \vdots &   & \DD_{\prows \pcols } &  & & \DD_{ \prows \npcols} & \\ 
     \prow_{\npiv} &  &  &  & & &\\ \cline{2-7} 
    \end{array} 
	\end{align*}
The meaning of $\equiv$ is unambiguous in this case, since the row and column labels show  exactly how rows and columns of the matrix on the left must be permuted to obtain the matrix on the right.

\begin{lemma}
\label{lem:invertiblepivblock}
 The block submatrices $\MM_{\prows \pcols}$ and $\DD_{\prows \pcols}$ are invertible.
\end{lemma}
\begin{proof}
By Equation (\ref{eq_matchcolumn}), the matrix $\MM_{\prows \pcols}$ is a generalized permutation matrix.  Since $\MM = \Ri \DD \CC$ one argues, by induction, that each  non-pivot row of $\DD$ (working from the bottom up) can be deleted without reducing the rank of $\DD$.  A similar argument holds for columns, working left to right.  The (full-rank) matrix that remains is $\DD_{\prows \pcols}$.
\end{proof}

Now let
    \begin{align*}
        \pcols^* = (\pcolp_1, \ldots, \pcolp_k)
    \end{align*}
denote the sequence of matched \underline{row} indices arranged in the order of their corresponding column indices.  Unlike the sequence $\pcols = (\pcol_1, \ldots, \pcol_k)$, sequence $\pcols^*$ is \underline{not} monotone increasing, in general.

\begin{lemma}
\label{lem:atri}
 The matrix $(\Ri \DD)_{\pcols^* \pcols}$ is invertible and upper triangular (though possibly not \underline{uni}triangular):
    \begin{align*}
        (\Ri \DD)_{\pcols^* \pcols} = 
        \begin{array}{l |cccc|}
        \multicolumn{1}{r}{} & \multicolumn{1}{c}{\pcol_2} & \multicolumn{1}{c}{\pcol_1} & \multicolumn{1}{c}{\cdots} &  \multicolumn{1}{c}{\pcol_\npiv} \\  \cline{2-5}
        \pcolp_1 & * & * & \cdots & *     \\   
        \pcolp_2 &   & * &  &  * \\   
        \vdots & &   & \ddots  & \vdots    \\   
        \pcolp_\npiv &   &   &  & *     \\  \cline{2-5}     
    \end{array}  
    \end{align*}
\end{lemma}
\begin{proof}
Rearranging the U-match decomposition of $\DD$, we have $\Ri \DD = \MM \Ci$.  Since $(\Ci)_{\upto{n}\{ \pcol_p\}}$ vanishes below row $\pcol_p$, the product $(\MM \Ci)_{\upto{n}\{ \pcol_p\}}$ has support on the set of pivot row indices that pair with pivot columns $\pcol_q$, for $q \le p$.  This establishes triangularity.  To see that $(\Ri \DD)_{\prows \pcols}$ is invertible, observe that $(\Ri \DD)_{\pcols^* \pcols}$ and $(\Ri \DD)_{\prows \pcols}$ are equal up to permutation of rows, and 
$$(\Ri \DD)_{\prows \pcols} \cdot \CC_{\pcols \pcols} = \Ri_{\prows \upto{m}} \cdot \DD \cdot \CC_{\upto{n} \pcols} = \MM_{ \prows \pcols}$$ is a generalized permutation matrix.
\end{proof}

\begin{lemma} 
\label{lem:blockuppertriangularcombs}
Suppose that $\RR \MM = \DD \CC$ is a \emph{proper} U-match decomposition.  Matrices $\RR$ and $\CC$ then have the following permuted block structure.  Moreover,  $\RR_{\prows  \prows}$ and $\CC_{\pcols\pcols}$ are upper unitriangular.
	\begin{align}
    \begin{array}{c |ccc|ccc|}
    \multicolumn{1}{c}{}& \multicolumn{1}{c}{\nprow_1} & \multicolumn{1}{c}{\cdots} & \multicolumn{1}{c}{\nprow_{m-k}} &  \multicolumn{1}{c}{ \prow_1 } & \multicolumn{1}{c}{\cdots} & \multicolumn{1}{c}{\prow_\npiv} \\  \cline{2-7}
     \nprow_1 &  &  &  & &  &  \\   
     \vdots &  & \id &  &  & \RR_{\nprows  \prows} &\\ 
     \nprow_{m-k} &  &  & &  & &\\ \cline{2-7}
     \prow_1 &  &  &  & & &\\ 
     \vdots &   & 0 &  & & \RR_{\prows  \prows} & \\ 
     \prow_\npiv &  &  &  & & &\\ \cline{2-7}     
    \end{array} 
    &&
    \begin{array}{c |ccc|ccc|}
    \multicolumn{1}{c}{}& \multicolumn{1}{c}{\pcol_1} & \multicolumn{1}{c}{\cdots} & \multicolumn{1}{c}{\pcol_\npiv} &  \multicolumn{1}{c}{ \npcol_1 } & \multicolumn{1}{c}{\cdots} & \multicolumn{1}{c}{\npcol_{n-k}} \\  \cline{2-7}
     \pcol_1 &  &  &  & &  &  \\   
     \vdots &  & \CC_{\pcols\pcols} &  &  & \CC_{\pcols\npcols} &\\ 
     \pcol_\npiv &  &  & &  & &\\ \cline{2-7}
     \npcol_1 &  &  &  & & &\\ 
     \vdots &   & 0 &  & & \id & \\ 
     \npcol_{n-k} &  &  &  & & &\\ \cline{2-7}     
    \end{array}    
	\label{eq:operationblocks}    
	\end{align}
\end{lemma}

\begin{proof}
Blocks $\RR_{\prows \prows}$ and $\CC_{\pcols \pcols}$ are upper triangular because $\prows$ and $\pcols$ are strictly increasing sequences.
The columns of $\RR$ indexed by $\nprows$ are standard unit vectors, by Axiom \ref{item:proper1} of proper U-match decomposition.  Likewise, the rows of $\CC$ indexed by $\nprows$ are standard unit vectors by Axiom \ref{item:proper2} of proper U-match decomposition.  
\end{proof}

In permuted block form, therefore, a {proper} U-match decomposition $\RR \MM = \DD \CC$ becomes
    \begin{align}
    \begin{array}{l |cc|}
        \multicolumn{1}{c}{}& \multicolumn{1}{c}{\nprows} & \multicolumn{1}{c}{\prows}  \\  \cline{2-3}
        \nprows & \id & \RR_{\nprows\prows}   \\   
        \prows &  &  \RR_{\prows  \prows} \\ \cline{2-3}     
    \end{array}     
    &&    
    \begin{array}{l |cc|}
        \multicolumn{1}{c}{}& \multicolumn{1}{c}{\pcols} & \multicolumn{1}{c}{\npcols}  \\  \cline{2-3}
        \nprows &  &    \\   
        \prows & \MM_{\prows  \pcols} &   \\ \cline{2-3}     
    \end{array}     
    &&
    =
    &&
    \begin{array}{l |cc|}
        \multicolumn{1}{c}{}& \multicolumn{1}{c}{\pcols} & \multicolumn{1}{c}{\npcols}  \\  \cline{2-3}
        \nprows & \DD_{\nprows  \pcols} & \DD_{\nprows  \npcols}   \\   
        \prows & \DD_{\prows  \pcols} &  \DD_{\prows  \npcols} \\ \cline{2-3}     
    \end{array}     
    &&
    \begin{array}{l |cc|}
        \multicolumn{1}{c}{}& \multicolumn{1}{c}{\pcols} & \multicolumn{1}{c}{\npcols}  \\  \cline{2-3}
        \pcols & \CC_{\pcols\pcols} & \CC_{\pcols\npcols}   \\   
        \npcols &  &  \id \\ \cline{2-3}     
    \end{array}   
    \label{eq:umatchblocks}    
	\end{align}    
\noindent where blank entries represent 0.  

\begin{remark}
We use the symbols $\prows$ (matched row indices) and $\nprows$ (unmatched row indices) to label the rows and columns of the square that represents matrix $\RR$ in Equation \ref{eq:umatchblocks}.  Note that the symbol for matched rows appears below and to the left of the symbol for unmatched rows.  This pattern reverses for matrix $\CC$: the symbol for matched indices, $\pcols$, appears above and to the left of the symbol for unmatched indices, $\npcols$.  
This counter-intuitive notational convention is in fact highly natural in the context of anti-transpose symmetry, 
c.f. Remark \ref{rmk:pivotnonpivotorder}.
\end{remark}

\subsection{Inner identities}
\label{sec:inneridentities}

The following identities demonstrate that every part of a proper U-match decomposition $(\RR, \MM, \DD, \CC)$ can be recovered from $\DD$ and the block submatrices $\MM_{\prows \pcols}$ and  $\Rirr$.  In  \S\ref{sec:lazylookup},  we will use these identities to prove correctness of a lazy look-up scheme that solves for any row or column of $\RR, \Ri, \CC$ or $\Ci$ via a single application of back-substitution.  Given their special role, we therefore refer to Equations\footnote{These are note, in fact, equations but equivalences, as discussed in \S\ref{sec:notation}.}  \eqref{eq:inneridentity1} - \eqref{eq:inneridentity4} as  \emph{inner identities}.

\begin{theorem}(Inner identities)
\label{thm:umatchblockidentities}
Posit a proper U-match decomposition $\RR \MM = \DD \CC$, and let $\AA = \Rirr \DD_{\prows  \pcols} = (\Ri \DD)_{\prows \pcols}$.  Then the  following matrix identities hold, where
    \begin{enumerate}
        \item blank entries indicate zero blocks
        \item permutations on the rows and columns of $\RR, \Ri, \CC, \Ci, \Ri \DD,$ and $\DD \CC$ are indicated by the sequences $\prows, \nprows, \pcols, \npcols, \upto{m},\upto{n}$, which appear as labels for blocks of row and column indices.
    \end{enumerate}
\begin{alignat}{2}
    &
    \CC 
    \;
    &&\equiv 
    \;
    \begin{array}{l |cc|}
        \multicolumn{1}{c}{}& \multicolumn{1}{c}{\pcols} & \multicolumn{1}{c}{\npcols}  \\  \cline{2-3}
        \pcols & \Ai \MM_{\prows \pcols} &  - \Ai \Rirr \DD_{\prows \npcols}        \\   
        \npcols &  &  \id \\ \cline{2-3}     
    \end{array}     
    \label{eq:inneridentity1}    
    \\[1em]
    &
    \Ci
    \;
    &&\equiv 
    \;
    \begin{array}{l |c|l}
        \multicolumn{1}{c}{}
        &
        \multicolumn{1}{c}{\upto{n}} 
        &  \multicolumn{1}{c}{} \\  \cline{2-2}
        \pcols & \MM_{\prows \pcols}^{-1} \Rirr \DD_{\prows  \upto{n}}   \\ 
        \cline{2-2}
        \npcols & \id_{\npcols \upto{n} }
        &   \\ \cline{2-2}    
    \end{array}   
    \label{eq:inneridentity2}     
    \\[1em]
    &
    \Ri 
    \;
    &&\equiv 
    \;
    \begin{array}{l |cc|}
        \multicolumn{1}{c}{}& \multicolumn{1}{c}{\nprows} & \multicolumn{1}{c}{\prows}  \\  \cline{2-3}
        \nprows & \id &  -\DD_{\nprows  \pcols} \DD_{\prows \pcols}^{-1}   \\
        \prows &  &  \Rirr \\ \cline{2-3}     
    \end{array}  
    \label{eq:inneridentity3}     
    \\[1em]
    &
    \RR
    \;
    &&\equiv 
    \;
    \begin{array}{l |c|c|}
        \multicolumn{1}{c}{}& \multicolumn{1}{c}{\nprows} & \multicolumn{1}{c}{\prows}  
        \\  \cline{2-3}
         \upto{m} 
        &
        \id_{\upto{m} \nprows} 
         &
         \DD_{\upto{m}  \pcols} \Ai
         \\ \cline{2-3}     
    \end{array}          
    \label{eq:inneridentity4}
\end{alignat}
\end{theorem}

\begin{proof} The proof is deferred to Appendix \ref{sec:proofofmatrixidentities}.
\end{proof}

\section{Compressed storage and lazy retrieval}
\label{sec:lazyumatch}

Here we present an effective approach to the following problem: given a U-match decomposition $\RR \MM = \DD \CC$, how can one store the associated matrices with as little memory as possible?  That is, how efficiently can one compress $\RR, \Ri, \CC, \Ci, \DD, \MM$ in storage, while maintaining the ability to quickly read any row or column, when it is needed?

On the one hand, $\DD$ and $\MM$ offer few opportunities for compression:  since we assume that  $\DD$ will be provided as input, its data structure is fixed \emph{a-priori}; $\MM$ has at most one nonzero entry per row and column, so  it presents relatively little to compress.  On the other hand, matrices $\RR, \Ri, \CC, \Ci$, prove excellent candidates for compression. Our storage scheme records \underline{\smash{none}} of these  matrices.  Rather, this approach
    \begin{enumerate}
        \item stores only three arrays in memory: $\DD, \MM$, and $\Rirr$;
        \item reconstructs the rows and columns of $\RR, \Ri, \CC, \Ci$ in a lazy fashion.
    \end{enumerate}

Under this scheme, \underline{\smash{retrieval is fast}}: any row or column can be reconstructed within a constant scalar multiple of matrix-vector multiplication time (Theorem \ref{thm:reconstructionisfast}).  The reconstruction procedure consists of sparse vector concatenation, permutation, matrix-vector multiplication, and at most one sparse triangular solve.

Moreover, \underline{\smash{compression is effective in reducing memory use}}:  numerical experiments show that $\Rirr$ holds orders of magnitude fewer nonzero entries than $\Ri$ in real-world applications, c.f.\ \S\ref{sec:experiments}.  Dropping not only $\Ri$, but $\RR, \Ci$, and $\CC$ as well, therefore offers substantial memory advantages.

\subsection{Data formats}

We will assume that the factored array,  $\DD$, is stored via a primitive data structure that permits $O(1)$ access to both rows and columns.  No further constraints are imposed.  This is consistent with current state of the art methods in persistent homology, where lazy methods can quickly build rows or columns of a boundary matrix.  Such data structures are increasingly available for a wide variety of common chain complex structures.

We will similarly assume that $\MM$ and $\Rirr$ admit $O(1)$ access to rows and columns; in applications where only rows (respectively, columns) are needed, one can drop the assumption of $O(1)$ access to columns (respectively, rows).

\begin{remark}
In a regime that calls for both rows and columns, one can, in the worst case, store two copies each of $\MM$ and $\Rirr$, one in CSR format and the other in CSC.  Such a strategy doubles memory use, but it is only needed in scenarios where the alternative would be to store two copies of each of the larger matrices $\Ri, \RR, \Ci, \CC$.
\end{remark}

\subsection{Lazy retrieval of rows and columns}
\label{sec:lazylookup}

There are 16 functionally distinct types of look-up request one might wish to perform on the matrices of a U-match decomposition, each  corresponding to one element of the following Cartesian product
    \begin{align}
        \{\text{pivot index, non-pivot index}\} \times  \{\text{row, column}\} \times \{\RR, \Ri, \CC, \Ci\}
        \label{eq:lazyparameters}
    \end{align}

 \begin{remark}
 \label{rmk:64cases}
  In fact the relevant number is closer to 64, since for purposes of lazy computation it is relevant to request a $\{$row, column$\}$ with entries sorted in $\{$ascending, descending$\}$ order according to the linear order on $\{$pivot rows, pivot columns$\}$.  Such a discussion is beyond the scope of this work, but merits systematic investigation.
 \end{remark}
    
The main result of this section is Theorem \ref{thm:trisolvebounds}. 
The proof of this result follows almost directly from Theorem \ref{thm:umatchblockidentities}, if we keep in mind the following:
\begin{enumerate}
    \item We can evaluate $T^{-1}b$ and $cT^{-1}$ by solving $Tx = b$ or $yT = c$, directly, for any invertible triangular matrix $T$.  Each one of these problems counts as one application of back-substitution.

    \item The rows and columns of matrix $\AA = \Rirr \DD_{\prows, \pcols}$ can be evaluated in a lazy fashion, since we assume efficient access to the rows and columns of $\Rirr$ and $\DD$.  Matrix $\AA$ is upper-triangular up to permutation of rows, by Lemma \ref{lem:atri} since $\pcols^* = \prows$ up to reordering.  Thus $\AA x = b$ and $y\AA =c$ can each be solved with a single application of  back-substitution.
    
    \item Matrix $\Drk$ factors as the product $\Ri_{\prows \prows} \AA$.  We have efficient access to the rows and columns of $\AA$ and $\Rirr$, so problems $\Drk b$ and $y \Drk  = c$ each count for one back-substitution.
\end{enumerate}

Table \ref{table:trisolves} catalogs the solve operation required for each lookup request.  The problems selected for back-substitution in this table can seem quite out of place at first, however they can be explained by further observation:

\begin{enumerate}\addtocounter{enumi}{3}
    \item Suppose that $T$ is an upper triangualr matrix and $S$ is an array of equal size.  Then to evaluate $\col_i(T^{-1} S)$ one should solve $T s =  \col_i(S)$.  However, to evaluate $S T^{-1}$ one should first solve $Tx = e_i$ to obtain the $i$th column of $T^{-1}$, then multiply this column by $S$ on the left.
\end{enumerate}
For example, retrieving either column $\pcol_i$ of $\CC$ or $\prow_i$ of $\RR$ will involve the matrix $\Ai$, if we follow Theorem \ref{thm:umatchblockidentities}.  However, executing the look-up for $\col_{\pcol_i}(\CC)$ requires a solution to $\AA x = \row_i(\MM_{\prows \pcols})$, while the executing a look-up for $\col_{\prows_i}(\RR)$ requires a solution to $\AA x = e_i$.  This disparity arises from the fact that for $\CC$ we multiply $\Ai$ with $\MM_{\prows \pcols}$ on the right, while for $\RR$ with multiply $\Ai$ with $\DD_{\upto{m} \pcols}$ on the left.

\begin{theorem}
\label{thm:trisolvebounds}
Let $\RR \MM = \DD \CC$ be a proper U-match decomposition.  If we have direct access to $\Rirr$,  $\MM$, and $\DD$, then
    \begin{enumerate}
        \item \emph{We may obtain the following with no triangular solves}: \\
        any row or column of $\Ci$, and any row of  $(\Ri)_{\prows \upto{m}}$
        \item \emph{We may obtain the following with at most one triangular solve}: \\
        any row or column of $\CC$ or $\RR$, and any column of $\Ri$
    \end{enumerate}
\end{theorem}
\begin{proof}
The specific triangular solve operation required for each of the look-ups are reported in Table \ref{table:trisolves}.  See the preceding discussion for a full explanation of how this table was generated.
\end{proof}

\begin{theorem}
\label{thm:reconstructionisfast}
Given access to $\DD$, $\MM$, and $\Rirr$, one can reconstruct any row or column of $\RR, \Ri, \CC$, or $ \Ci$ in $O(m n)$ time, where $\DD \in \field^{m \times n}$.
\end{theorem}
\begin{proof}
 Theorems \ref{thm:umatchblockidentities} and \ref{thm:trisolvebounds} imply that any row or column of $\RR, \Ri, \CC$, or $ \Ci$ can be reconstructed via (1) some matrix-vector multiplications, (2) concatenation of at most one pair of vectors, and permuting their entries, and (3) at most one sparse triangular solve operation.  All matrices involved have size $m \times n$ or smaller.
\end{proof}

\begin{center}
\begin{table}
\begin{tabularx}{\textwidth}{lYY}
  \toprule
     $\row$ &$\prow_i \; / \; \pcol_i$&$\nprow_i \; / \; \npcol_i$\\
    \midrule
    $\Ri$&$(\dagger \dagger)$&$x \Drk = -\row_{i}( \DD_{\nprows  \pcols})$\\
    $\RR$&$x \Rirr = e_i$&$x \AA = \row_{i}( \DD_{\nprows \pcols})$\\
    $\Ci$&$(\dagger)$&$e_{\npcol_i}$\\
    $\CC$ & $x \AA = e_i$&$e_{\npcol_i}$\\
    \midrule
    
    $\col$&$\prow_i \; / \; \pcol_i$ & $\nprow_i \; / \; \npcol_i$\\
    \midrule
    $\Ri$&$\Drk x = e_i$&$e_{\nprow_i}$\\
    $\RR$&$\AA x = e_i$&$e_{\nprow_i}$\\    
    $\Ci$& $(\dagger)$
    &$(\dagger)$\\    
    $\CC$&$\AA x =  \col_{i} \MM_{\prows \pcols}$&$\AA x =  - \Rirr \col_{i}(\DD_{\prows \npcols})$\\
    \bottomrule
    \end{tabularx}
    \caption{
        Computations necessary to obtain rows and columns of matrices assuming access to only $\Rirr$, $\MM$, and $\DD$.  If a look-up requires the solution of a linear equation via back substitution, then this problem is expressed in form $T x = b$ or $yT = c$.  If no triangular solve is necessary because the vector in question can be computed by reindexing and sparse matrix vector multiplication, the corresponding entry is marked with a $(\dagger)$.  If the vector in question is the $i$th standard unit vector, then the corresponding entry is marked as $e_i$.  The entry for pivot rows of $\Ri$ is marked $(\dagger \dagger)$; in this special case no algebraic operations whatsoever are performed -- one only needs to permute the entries of the corresponding row of $\Rirr$, and insert some zeros.
    }    
    \label{table:trisolves}
\end{table}
\end{center}

\section{Lazy global duality}
\label{sec:gloabl_duality_with_umatch}
\label{sec:lazy_global_duality}

The notion of global dualtiy was introduced\footnote{See  \S\ref{subsec:computationalassymetry} for a review.} in \cite{de2011dualities}.  Roughly translated into the language of the current discussion, the critical observation from \cite{de2011dualities} states that each (0-graded) U-match decomposition of the total boundary operator $\DD$ contains ``all you need to know,'' to decompose a (relative) (co)homological persistence module.  

Theorem \ref{thm:jordanfromumatch} formalizes this fact.  The proof requires very little algebraic machinery,  thanks to the close connection between U-match and  right-reduction, i.e.\ $R = DV$ decomposition.

\begin{theorem}
\label{thm:jordanfromumatch}
If $\RR \MM = \DD \CC$ is a U-match decomposition,  then 
    \begin{align*}
        \underbrace{(\RR \MM) = \DD \CC}_{\rreduction}
         &&
        \underbrace{(  (\CC^{-1})^{\perp}\MM^\perp) = \DD^\perp (\RR^{-1})^{\perp}}_{\rreductionAlt}
    \end{align*}
are both right-reductions.  If, in addition,  $\RR$ and $\CC$ are both 0-graded, then  $\mbasisR{\rreduction}$ and $(\mbasisR{\rreductionAlt}^{\perp})^{-1}$ are (filtered, graded) Jordan bases of $\DD$.  

In particular, Theorem \ref{thm:globalduality} applies to $\mbasisR{\rreduction}$ and $(\mbasisR{\rreductionAlt}^{\perp})^{-1}$.  Consequently, a 0-graded U-match decomposition of $\DD$ provides all of the data necessary to decompose a (relative) (co)homological persistence module into indecomposable submodules.
\end{theorem}
\begin{proof}
That $\rreduction$ and $\rreductionAlt$ are right-reductions follows from Lemma \ref{lem:umatchRDVgeneral} and anti-transpose symmetry.  The conclusion therefore follows directly from Theorem \ref{thm:globalRDV}.
\end{proof}

Theorem \ref{thm:jordanfromumatch} is typical of results concerning U-match decomposition and persistence.  On the one hand, much of the theoretical heavy lifting has already been established in the literature.  On the other hand, the U-match formalism recasts these results in a manner that  clarifies concepts and facilitates theorem-proving.

Theorem \ref{thm:umatchsaecularbasis} below illustrates this point par excellence; the same result could be proved in any number of ways, but U-match decomposition allows a concise account that clarifies the underlying concepts.  Another important example is the family of applications that motivate the present work: lazy algorithms for global duality.  Work with such algorithms on the level of sparse vectors and indices can be painstaking and laborious.  However, theorems of a highly practical nature (e.g., to algorithm design) can be directly proved via the block structure described in Proposition \ref{thm:umatchblockidentities} and elsewhere.

\begin{remark}
\label{rmk:scale2umatch}
Theorem \ref{thm:jordanfromumatch} implies that  $\MM = \mbasisR{\rreduction}^{-1} \DD \mbasisR{\rreduction}$ is a generalized matching matrix.  Thus, we \underline{almost} have a U-match decomposition $\mbasisR{\rreduction} \MM = \DD \mbasisR{\rreduction}$.  It is not a true U-match decomposition in general, since $\mbasisR \rreduction$ may have entries other than 1 on the diagonal.  

This can be corrected by multiplying $\mbasisR{\rreduction}$ on the right with an invertible diagonal matrix $\diagmat$, yielding a bona fide U-match decomposition 
    \begin{align*}
        (\mbasisR{\rreduction} \diagmat)^{-1} \MM = \DD (\mbasisR{\rreduction} \diagmat).
    \end{align*}  
However, the columns of  $\mbasisR{\rreduction} \diagmat$ may no longer form a Jordan basis of $\DD$.  Indeed, one can find $\diagmat$ satisfying both the condition that $\mbasisR{\rreduction} \diagmat$ has 1's on the diagonal and the condition that $\mbasisR{\rreduction} \diagmat$ is a Jordan basis if and only if all diagonal elements of $\mbasisR{\rreduction}$ are equal.
\end{remark}

\subsection{Bases for cycles, boundaries, and the saecular lattice}

\newcommand{\upfrom}[1]{\overset{*}{\mathbf{#1}}}
\newcommand{\npcolsAT}{Y^{\perp c}} 
\newcommand{\pcolsAT}{Y^{\perp}} 

\newcommand{\pairset}{\Xi} 

\newcommand{\saecularlattice}{\mathscr S}

Let $\subspacelattice$ denote the order lattice of subspaces of  $\Chains$.  Let $\saecularlattice$ denote the \emph{saecular lattice} for $\Chains$, the sublattice of $\subspacelattice$ generated by all subspaces that can be expressed in one of the following three forms
    \begin{align*}
        \filta_\fparam \Chains_\simplexdim
        &&
        \DD_\bullet (\filta_\fparam \Chains_\simplexdim)
        &&
        \DD^\bullet (\filta_\fparam \Chains_\simplexdim)
     \end{align*}
for some $\simplexdim$ and $\fparam \ge 0$.  Note, in particular, that $\saecularlattice$ contains $\filta_\fparam\Cycles_\simplexdim$ and $\filta_\fparam\Boundaries_\simplexdim$, as well as $\partial \Ri\filta_\fparam\Chains_\simplexdim$.

\begin{theorem}
\label{thm:umatchsaecularbasis}
Let $\RR \MM = \DD \CC$ be a U-match decomposition.  Suppose $\RR$ and $\CC$ are 0-graded, and define a right-reduction $\rreduction$ as in Theorem \ref{thm:jordanfromumatch}.  Then the columns of $\mbasisR{\rreduction}$ contain a basis for each element of $\saecularlattice$.
\end{theorem}
\begin{proof}
There exists a diagonal matrix $\diagmat$ such that   $(\mbasisR{\rreduction} \diagmat)^{-1} \MM = \DD (\mbasisR{\rreduction} \diagmat)$ is a bona-fide U-match decomposition, as discussed in Remark \ref{rmk:scale2umatch}.

The columns of $\mbasisR{\rreduction}$ therefore contain bases for every subspace of the form $\DD_\bullet (\filta_\fparam \Chains_\simplexdim)$ and $\DD^\bullet (\filta_\fparam \Chains_\simplexdim)$, by Theorem \ref{thm:bifiltrationbasis}, statements 1 and 2.  The columns of $\mbasisR{\rreduction}$ also contain a basis for each $\filta_\fparam \Chains$, since  $\mbasisR{\rreduction}$ is invertible and upper triangular.  Using Lemma \ref{lem:subspace_bases} we can then construct the desired bases from these.
\end{proof}

\begin{remark}
As discussed in Section 
\ref{sec:fundamentalsubspaces}, identifying the subset of columns of $\mbasisR{\rreduction}$ that freely generate a given subspace in $\saecularlattice$ requires very little effort, and can be deduced from the sparsity pattern of $\MM$.
\end{remark}

\begin{remark}
The U-match decomposition $\MM = (\mbasisR{\rreduction} \diagmat)^{-1} \DD (\mbasisR{\rreduction} \diagmat)$ need not be proper, in general.  
\end{remark}

\begin{remark}
Theorem \ref{thm:umatchsaecularbasis} has a global dual.  One replaces $\filta_\fparam \Chains_\simplexdim$ (a space which one can regard as the image of the inclusion $\filta_\fparam \Chains_\simplexdim \su \filta_\fparammax \Chains_\simplexdim$) with the image of the inclusion $ (\Chains/\filta_\fparam \Chains)^\simplexdim \su  (\Chains/\filta_\fparammin \Chains)^\simplexdim = \Chains^\simplexdim$.  This generates a globally dual saecular lattice $\saecularlattice^*$, whose elements are  freely generated by the rows of matrix $\mbasisR{\rreductionAlt}^\perp$, where $\rreductionAlt$ is defined as in Theorem \ref{thm:jordanfromumatch}. 
\end{remark}

\subsection{Lazy access to Jordan bases}

\newcommand{\RRalt}{{\mathfrak L}}

 U-match factorization provides the following two-step recipe for computing a filtered Jordan basis of the filtered differential operator $\DD$: 
 
 \begin{enumerate}
     \item  Apply Algorithm \ref{alg:revised_lrdec} to obtain an invertible submatrix $\Rirr$ for a proper U-match decomposition $\RR \MM = \DD \CC$.  
     \item Apply the look-up and retrieval methods described in \S\ref{sec:lazyumatch} to access the columns of $\CC$ in a lazy fashion; these columns can then be translated into a Jordan basis, via Theorem \ref{thm:jordanfromumatch}. 
 \end{enumerate}

This is by no means the only lazy approach to constructing Jordan bases; several variants on this method, including optimizations designed to accelerate computation and sparsify output, can be found in Appendix \ref{sec:lazy_jordan_alt}.

\subsection{Linear and inverse problems}
\label{sec:exampleproblems}

A surprising diversity of direct and inverse problems arise in the natural course of applications with persistent (relative) (co)homology.  One of the attractive features of U-match decomposition is a unified framework to address these.  

\subsubsection*{Homology}

\newcommand{\bb}{b}
\newcommand{\chainc}{f}

Let $\chaina \in \Chains_\simplexdim$ be a chain of dimension $\simplexdim$.  The \emph{birth time} of $\chaina$ is the moment at which $\chaina$ enters the filtration, i.e.\ $\min \{\fparamb : \chaina \in \filta_\fparamb \Chains_\simplexdim\} $.  The \emph{bounding time} of $\chaina$ is the moment at which $\chaina$ becomes nullhomologous, or $\infty$, if $\chaina$ never becomes nullhomologous. Concretely, this translates to $\min \{ \fparamb : \chaina \in  \filta_\fparamb\Boundaries_\simplexdim\}$, if we take $\min \emptyset = \infty$.  The \emph{lifespan} of $\chaina$ is the half-open interval $[a,b)$, where $a$ and $b$ are the birth and bounding time of $\chaina$, respectively.

One can compute the bounding time of $x$ by solving the \emph{earliest bounding chain problem}.  This is the inverse problem of finding a specific element $\chainb \in \filta_\fparam \Chains_{\simplexdim+1}$ such that $\chaina = \DD \chainb$, where $\fparam$ is the bounding time of $\chaina$.  More generally, one could solve for the affine space of \underline{all} solutions to $\chaina = \DD \chainb$ in $\filta_\fparamb \Chains_{\simplexdim+1}$, for each $\fparamb \ge \fparam$.  Finally, given a different $\simplexdim$-chain, $\chainc$, one can ask where, in the filtration, $\chaina$ and $\chainc$ become homologous.  We call this the \emph{time of homology} problem.

All these problems become easy to solve when one is given access to a (0-graded) U-match decomposition of $\DD$.  First, we may apply Proposition \ref{prop:solvingsystems} either (i) to decide that no bounding chain exists, in which case the space of solutions is empty and the bounding time of $x$ is $\infty$, or (ii) to obtain a solution $\chainb$ to $\DD \chainb = \chaina$.  Corollary \ref{cor:minmaxsolutions} then ensures that $\chainb$ is an earliest bounding chain.  In this case, the space of bounding chains at time $\fparamb \ge \fparam$ can be expressed in form $\chainb + \filta_\fparamb\Cycles_{\simplexdim+1}$, and an explicit basis for $\filta_\fparamb\Cycles_{\simplexdim+1}$ can be computed as in \S\ref{sec:fundamentalsubspaces}. Technically this result indicates that the columns of $\CC$ contain a basis for the space of \underline{all} cycles at time $\fparamb$; however, since we assume that all matrices are 0-graded one can show that the subset of columns representing chains of dimension $\simplexdim+1$ is a basis for $\filta_\fparamb\Cycles_{\simplexdim+1}$.

The birth time of $\chaina$ requires no special machinery to compute; if $\chaina$ is a linear combination of basis vectors $\sum_{j \in J} \alpha_j \svec_j$ with each $\alpha_j \neq 0$, then $\chaina$ first appears at time $t = \max J$.  The lifespan of $\chaina$ is obtained for free once one calculates the birth and bounding times.

Finally, one can solve the time of homology problem by applying the methods already discussed to the difference $\chaina - \chainc$, since $[\chaina] = [\chainc]$ in $\Homologies_\simplexdim(\filta_\fparamb \Chains)$ precisely when $\chaina, \chaina - \chainc \in \filta_\fparamb \Chains$ and $\chaina - \chainc  \in \filta_\fparam\Boundaries_\simplexdim$.

\subsubsection*{Relative cohomology}

The discussion of inverse problems for homology has a global dual for persistent relative cohomology.  There is a natural sense in which this dual is ``just'' anti-transpose symmetry. However, there are several interesting differences in semantic interpretation.  For example, it is easy to determine the lefthand endpoint of the time interval where $\chaina$ represents a nonzero homology class and harder to determine the righthand endpoint; indeed, sufficiently hard to oblige us to use U-match machinery to recover it. For persistent relative cohomology, the situation is reversed; that is, the righthand endpoint is easy to determine and the lefthand requires work. This is, of course, a reflection of the fact that the anti-transpose operation reverses order of rows and columns. In the interest of space, we omit further details.  

\subsubsection*{Change of basis}

Proposition \ref{prop:kerneldeletion} makes it straightforward to perform change of basis operations on cycles and cocyles.  In particular, if we wish to re-express a column vector $c$ as a linear combination in the columns of domain \COB{} $\RR$, then we may compute $\Ri c$ by setting some entries of $c$ to zero.

\section{Factorization algorithms}
\label{sec:factorization_algorithms}

Here we present two variants of a familiar algorithm, Gauss-Jordan elimination. One variant returns a full U-match decomposition of form $\Ri \DD = \MM \Ci$.  The other returns only $\Rirr$, which, as we saw in \S\ref{sec:lazyumatch}, suffices to reconstruct the other parts of the decomposition, if and when they are needed.   

Readers versed in computational persistence may observe that Algorithm \ref{alg_lrdec} is essentially the \emph{cohomology algorithm}, a standard method to compute persistent relative cohomology \cite{de2011dualities}.  Up to relabeling, this method amounts to an application of the ``standard algorithm'' \cite{cohen2006vines} to the antitranspose of $\DD$.  
The procedure is ``left-looking,'' in the sense that no information about row $i$ of $\DD$ is needed until we construct row $i$ of $\Ri$.

We write $\id^{n\times n}$ for the $n \times n$ identity matrix and $0^{m \times n}$ for the $m \times n$ zero matrix.

\medskip

\begin{algorithm}[H]
\SetAlgoLined
\KwIn{Matrix $D\in \field^{m\times n}$.}
\KwResult{Upper unitriangular matrices $\Ri\in\field^{m\times m},\Ci\in\field^{n\times n}$ and matching matrix $M\in\field^{m\times n}$ such that $\Ri D=M\Ci$.  The corresponding U-match decomposition is $\RR M = \DD \CC$.}
Initialization: $\Ri\gets I^{m\times m}$, $\Ci \gets I^{n\times n}$, $M\gets 0^{m\times n}$ \;
\For{$i\gets m$ \textbf{to} $1$}{
\While{ there exist  $j \in \{i+1, \ldots, m\}$ and $k\in \{1, \ldots, n\}$ such that  $\row_i(\DD)$ and $\row_j(\DD)$ both have leading nonzero entries in column $k$, } 
{
$\row_i(\DD) \gets \row_i(\DD) - \frac{D[i,k]}{D[j,k]} \row_j(\DD)$\;
$\row_i(\Ri) \gets \row_i(\Ri)-\frac{D[i,k]}{D[j,k]} \row_j(\Ri)$\;
}
}
\For{$i\gets 1$ \textbf{to} $m$}{
\If{for some $k$, $D[i,k]$ is the leading entry of $\row_i(\DD)$}{
$\row_k(\Ci)\gets \frac{1}{D[i,k]}\row_i(\DD)$\;
$M[i,k] \gets D[i,k]$\;
}
}
\caption{Proper U-match decomposition (uncompressed)}
\label{alg_lrdec}
\end{algorithm}

\begin{proposition}
\label{prop:correctnessofalgorithm1}
Algorithm \ref{alg_lrdec} returns a U-match decomposition.
\end{proposition}
\begin{proof}
For any $\row_i(\DD)$, the first nested loop in Algorithm~\ref{alg_lrdec} adds scaled rows $\row_j(\DD)$ below $\row_i(\DD)$ to clear the first non-zero entry of $\row_i(\DD)$; while $\Ri$ records the row operations during the clearing process. Hence, $\Ri$ is an upper unitriangular matrix and $\Ri \DD$ becomes a reduced matrix such that for any two non-zero rows of $\Ri \DD$, the first non-zero entries have different column indices. 

The second loop scales each non-zero row of $\Ri \DD$ and reassembles them to form an upper unitriangular matrix $\Ci$. 

To verify $\Ri \DD=M\Ci$ or $LM=DR$, observe that if the $i$-th row of $\Ri \DD$ is $0$, then the $i$-th row of $M$ (and also $M\Ci$) is zero; if the $i$-th row of $M$ is not $0$, assume that $M[i,k]\neq 0$, then $\row_i(M\Ci)=\row_i(M)\Ci=M[i,k]\row_k(\Ci)=D[i,k]\row_k(\Ci)=\row_i(\Ri \DD)$.
\end{proof}

Since our work is motivated by applications where $\Ri$ is too large to store in memory, one would naturally like an analog of Algorithm \ref{alg_lrdec} which does not construct  $\Ri$ completely, but rather just the submatrix $\Rirr$.  Algorithm \ref{alg:revised_lrdec} is one such method.  

Significantly,  Algorithm \ref{alg:revised_lrdec} does not record any of the modified rows of $\DD$ in memory; rather it recomputes each modified pivot row as the product of a corresponding row of $\overline{\RR}$ with $\DD$, whenever needed.  This is essentially the strategy pioneered by \cite{bauer2019ripser}; as discussed in that work, it can cut memory use dramatically.

\medskip

\begin{algorithm}[H]
\SetAlgoLined
\KwIn{Matrix $\DD \in \field^{m\times n}$.}
\KwResult{Matching matrix $\MM\in\field^{m\times n}$ and upper triangular matrix $\overline{\RR}$ such that $\overline{\RR} = \Rirr$, where $\Ri$ is the row operation matrix returned by Algorithm \ref{alg_lrdec}.}
Initialization: $\overline{\RR}\gets \emptyset$, $\MM\gets 0_{m\times n}$, $\mathrm{indices} = \emptyset$\;
\For{$i\gets m$ \textbf{to} $1$}{
$\mathrm{vec} =[0,0,\cdots,0]\in \field_{1\times (m-i)}$\;
$\mathrm{row} = \row_i(\DD)$ \;
$\mathrm{end} = \mathrm{indices}.\mathrm{length}()$ \;

\While{ exist $j=\mathrm{indices}[l]$ and $k$ with $M[j,k]\neq 0$ and $k$ is the first nozero entry of $\mathrm{row}$,} 
{
$\lambda = \mathrm{row}[k]/\MM[j,k]$\;
$\mathrm{reduced} = \row_l(\overline{\RR})\cdot D_{\mathrm{indices}, \upto{n}}$ \;
$\mathrm{row} \gets \mathrm{row} - \lambda\cdot \mathrm{reduced}$\;
$\mathrm{vec} \gets \mathrm{vec} -\lambda\cdot \row_j(\Ri)$\;
}

\If{$\mathrm{row}\neq 0$}{
let $k$ be the first non-zero entry of $\mathrm{row}$ \;
$\overline{\RR}\gets
\begin{bmatrix}
1 & \mathrm{vec} \\
0 & \overline{\RR}
\end{bmatrix}$\;
$M[i,k]=D[i,k]$\;
$\mathrm{indices}.\mathrm{push}(i)$\;
$\mathrm{indices}.\mathrm{sort}()$\;
}
}
\caption{Proper U-match decomposition, (compressed)}
\label{alg:revised_lrdec}
\end{algorithm}

\begin{proposition}
Let $\Ri$ and $\overline{\RR}$ be the matrices returned by Algorithms \ref{alg_lrdec} and  \ref{alg:revised_lrdec}, respectively, and let $ \rho=(\prow_1, \cdots, \prow_k)$ denote the indices of the pivot rows of the elimination procedure in Algorithm \ref{alg_lrdec}.  Then $\overline{\RR} =  (\Ri)_{\rho\rho}$.
\end{proposition}
\begin{proof}
The statement follows from the fact that Algorithm~\ref{alg:revised_lrdec} is just an adopted version of Algorithm~\ref{alg_lrdec} in which we only record the rows and columns of $\Ri$ whose indices corresponding the non-zero rows of $\Ri \DD$.
\end{proof}

\subsection{Short circuit optimizations specific to persistence}
\label{sec:shortcircuitph}

Several short-circuit optimizations from the current literature in PH computation can also be applied to optimize our {factorization} algorithms. In particular, \emph{clearing optimization} allows one to skip some iterations of the for loop, and \emph{Morse pairings} allow one to short circuit many iterations that can't be skipped by clearing.  Here we discuss these ideas in detail.

\subsubsection*{Clear and compress (computation with a twist)}

\emph{Clearing} and \emph{compression} refer to a family of acceleration techniques core to modern persistent (co)homology computation.  The mathematical basis for these techniques  can be traced to the following general property of matching relations on  2-nilpotent matrices.

\begin{proposition}
\label{prop:clearandcompress}
If $\IMatch$ is the matching relation of a square matrix $\DD$, and $\DD^2 = 0$, then $\DEF(\IMatch) \cap \VAL(\IMatch) = \emptyset$.
\end{proposition}
\begin{proof}
 Fix a U-match decomposition $\RR \MM = \DD \CC$, 
and define a right-reduction $\rreduction$  as in Theorem \ref{thm:jordanfromumatch}.  Then there exists another U-match decomposition, $(\mbasisR{\rreduction} \diagmat) \MM = \DD (\mbasisR{\rreduction} \diagmat)$, as discussed in Remark \ref{rmk:scale2umatch} (recall that the matching array, $\MM$, is the same in all U-match decompositions, c.f.\ Theorem \ref{thm:matchingunique}).  We have $\MM =  (\mbasisR{\rreduction} \diagmat)^{-1} \DD (\mbasisR{\rreduction} \diagmat)$, hence $\MM^2 = 0$.  Since $\MM$ is a matching array, this implies $\DEF(\IMatch) \cap \VAL(\IMatch) = \emptyset$.
\end{proof}

Proposition \ref{prop:clearandcompress} states, in particular, that the set of pivot rows in Algorithm \ref{alg:revised_lrdec} and the set of rows indexed by $\VAL(\IMatch)$ are disjoint.  In practice, this often means that one can ``ignore'' rows indexed by $\VAL(\IMatch)$ during matrix reduction.  

The idea behind this result, and its implications for PH computation, can be traced at least as far back as one of the earliest works in this field, \cite{ZCComputing05}.  Here the authors use a variant on Proposition \ref{prop:kerneldeletion} to show that certain change of basis operations on a boundary matrix $\partial_\simplexdim$ result in the zeroing out of pivot rows, leaving all other rows unchanged; they use this zeroing procedure as a preprocessing step in right reduction.  

Proposition \ref{prop:clearandcompress} also implies that the set of pivot columns in Algorithm \ref{alg:revised_lrdec} and the set of columns indexed by $\DEF(\IMatch)$ are disjoint.  This gives rise to a natural dual approach, in which one skips over, clears, or deletes columns indexed by $\DEF(\IMatch)$.  This idea was developed into a formal acceleration technique in \cite{CKPersistent11}; in experiments, the technique improved time and memory use substantially.

These ideas were further developed in \cite{BKRClear14}, which computes PH in ``chunks''; after a chunk is computed, information about $\DEF(\IMatch)$ and $\VAL(\IMatch)$ are extracted and used to simplify the boundary matrix.  

Algorithms that compute rows of $\DD$ in a lazy fashion need not delete or simplify rows of $\DD$ at all; they can simply skip over them.  This saves not only algebraic operations (which would otherwise have been needed to reduce a row to zero), but also the computations needed to construct that row in the first place.  The pairing of lazy methods with this clearing optimization was pioneered in \cite{bauer2019ripser}.  In numerical experiments, the time and memory saved by excluding rows indexed by $\VAL(\IMatch)$ is disproportionate to the number of rows excluded \cite{CKPersistent11, BKRClear14, bauer2019ripser, zhang2019hypha}, at least for \underline{\smash{clique complexes}}.  Effects for cubical complexes tend to be nontrivial but less pronounced \cite{bauer2014distributed}.  The reason for these empirical trends is a matter of ongoing research; see  \cite{bauer2019ripser} for a nice review.  New experimental evidence is reported in \S\ref{sec:experiments}.

In concrete terms, the clear/compress/twist optimization can be applied to Algorithm \ref{alg:revised_lrdec} by inserting the following lines at the very beginning of each iteration of the outer for-loop (that is, the loop that iterators over $i$):

\begin{algorithm}[H]
\SetAlgoLined
$\ldots$
\\
\For{$i\gets m$ \textbf{to} $1$}{
    \If{$M[p,i] \neq 0$ for some $p > i$}{
    \textbf{continue}
    }
$\cdots$
}
$\cdots$
\caption{The clear/compress/twist short circuit}
\label{alg:clearcompresstwistshortcircuit}
\end{algorithm}

\medskip

\subsubsection*{Minimal/steepness/emergent optimization}

Suppose that the $i$th iteration of the outer for-loop in Algorithm \ref{alg:revised_lrdec} corresponds to the $p$th matched row, meaning $i = \prow_p$.  Suppose, moreover, that either of the following two equivalent conditions holds true
    \begin{enumerate}[label=\textbf{(B\arabic*)}]
        \item \label{item:pareto1} The column index matched to row $\prow_p$ coincides with the leading nonzero entry of row $\prow_p$.  In symbols, $\prowp_p = \min \supp (\row_{\prow_p}(\DD))$.
        \item \label{item:pareto2} The first nonzero entry of row $\row_{\prow_p}(\DD)$ appears in column $j$, and $\col_j(\MM) = 0$ on this iteration of the outer for-loop.\footnote{If this condition is satisfied, then column $j$ of $\MM$ will \underline{become} nonzero at the end of iteration $i = \prow_p$, however.}  Concretely, this condition holds iff (i) $\DD[i,j] \neq 0$, (ii) $\DD[i,j'] = 0$ for $j' < j$, and  (iii) $\DD[i', j] = 0$ for $i < i'$.
    \end{enumerate}

\newcommand{\Par}{\mathrm{Par}}

We refer to the set of all pairs $(\prow_p,  \prowp_p) \in \IMatch$ such that \ref{item:pareto1} and \ref{item:pareto2} hold for $i = \prow_p$ as the set of \emph{Pareto pairs} of $\DD$.  This set is denoted $\Par(\DD)$.

Algorithm \ref{alg:revised_lrdec} performs no algebraic operations on any row $i$ such that $i = \prow_p$ for some $(\prow_p,  \prowp_p) \in \Par(\DD)$; this can be confirmed by a cursory examination of the procedure.  Rather, Algorithm \ref{alg:revised_lrdec} will simply set $\MM[i, \prowp_p]$ equal to $\DD[i, \prowp_p]$, and extend $\bar L$ to a matrix whose top row is $[1, 0, \ldots, 0]$.

 In a lazy regime where one constructs each row of $\DD$ on the fly, one can therefore short-circuit the construction of row $i$ as soon as (i) it has been determined that $i = \prow_p$ for some $(\prow_p,  \prowp_p) \in \Par(\DD)$, and (ii) index $\prowp_p$ and entry $\DD[\prow_p, \prowp_p]$ have been calculated.  

While elementary in principle, the implementation of this short-circuit technique can be complex in practice; some lazy constructors -- including many of the most commonly used constructors for filtered clique and cubical complexes -- do not generate the entries of each row in sorted order, so determining the first nonzero entry of row $i$ can entail a nontrivial computational cost.  Nevertheless, effective implementations do exist, and are regarded as essential to many of the fastest solvers currently available (at least, for computations involving clique complexes).

Unlike the short-circuit method for clear/compress/twist, this short-circuit method  cannot be easily expressed by inserting some additional lines into Algorithm \ref{alg:revised_lrdec}.  Rather, it calls for a modification to the lower-level subroutine that finds the first nonzero entry of row $i$ and then performs the associated update on $\bar \RR$ and $\MM$.  Thus, pseudocode is omitted.

\begin{remark}[Historical note]
The set $\Par(\DD)$ has received much attention over the past decade, and has been independently discovered by a variety of authors.  

Kahle \cite{kahle2011random} developed an important instance of this object in work on  certain probability spaces of simplicial complexes. Delgado-Friedrichs et al. considered an analogous construction for cubical complexes in \cite{delgado2014skeletonization}, which names the elements of $\Par(\DD)$ \emph{close pairs}.  Independently, Henselman-Petrusek defined $\Par(\DD)$ for arbitrary boundary matrices under the name \emph{Pareto frontier} in \cite{henselman2016matroid}, and later termed the elements of this set \emph{minimal pairs} \cite[Remark 8.4.2]{henselman2017}.  Bauer generalized the construction of Kahle to arbitrary filtered simplicial complexes for use in PH computations, in work that ultimately appeared in \cite{bauer2019ripser}.  Following the appearance of that work, Lampret independently developed a more general construction called a \emph{steepness matching}; unlike those which preceded it, the steepness matching is suitable for boundary matrices with coefficients in an arbitrary unital ring (by contrast, the predecessors restricted to field coefficients).

It has been further noted \cite{kahle2011random, bauer2019ripser, henselman2016matroid, lampret2020chain, delgado2014skeletonization} that $\Par(\DD)$ constitutes a discrete Morse vector field -- a fact with deeper  implications for both theory and algorithms.
\end{remark}

\section{Experiments}
\label{sec:experiments}

In \S\ref{sec:lazyumatch}-\ref{sec:factorization_algorithms}, we present a computational scheme for  U-match decomposition, storage, and retrieval of a matrix $\DD$.  While no restrictions are placed on $\DD$, the scheme is specifically chosen to work  with boundary operators of filtered chain complexes.  In particular, the scheme is optimized for applications where the following conditions hold:

\begin{enumerate}[label=\textbf{(C\arabic*)}]

    \item \label{item:matrixassumptioncompressed}  \emph{Algorithm \ref{alg:revised_lrdec} (compressed decomposition) returns smaller outputs than Algorithm \ref{alg_lrdec} (uncompressed decomposition). In particular,  $\Rirr$ holds substantially fewer nonzero entries than $\Ri$.}  

    \item \label{item:matrixassumptionantitranspose} 

    \emph{Algorithm  \ref{alg:revised_lrdec} requires less time and memory to decompose $\DD$  than $\DD^\perp$.}    

\end{enumerate}

Here we present numerical evidence that \ref{item:matrixassumptioncompressed} and \ref{item:matrixassumptionantitranspose} do hold for a broad range of filtered boundary operators $\DD$ found in topological data analysis, thus justifying our design decisions.  These experiments further support the thesis that

\begin{enumerate}[label=\textbf{(C\arabic*)}]
\setcounter{enumi}{2}
    \item \label{item:matrixassumptionlazy}  \emph{Compared to the standard alternative for computing cycle representatives in persistent homology -- right-reduction of $\DD$ -- the lazy U-match scheme  consumes substantially less time and memory.}
\end{enumerate}

\subsection{Design}

In order to test \ref{item:matrixassumptioncompressed} - \ref{item:matrixassumptionlazy}, we consider a set of real-world and simulated data sets.  Each data set engenders either a filtered simplicial complex or filtered cubical complex; in either case the complex becomes nullhomotopic by the last step in the filtration.

To each filtered complex, $X$, we associate a matrix $\DD$ representing the degree-2 boundary operator $\Chains_2(X, \binary) \to \Chains_1(X, \binary)$, where $\binary$ represents the Galois field of order 2.  The rows and columns of $\DD$ are sorted in ascending order of birth time.  

We pre-compute the index matching relation $\IMatch$ of $\DD$; this relation implicitly contains  the sequence of row-pivot elements $\prows = (\prows_1, \ldots, \prows_k)$ and column-pivot elements $\pcols = (\pcols_1, \ldots, \pcols_k)$.  We also pre-compute the number of off-diagonal nonzero entries in $\Ri$ and $\Rirr$, where $\Ri$ is the row operation matrix returned by Algorithm \ref{alg_lrdec}.  Finally, we measure the time and memory used by Algorithm \ref{alg:revised_lrdec} to decompose each of the following matrices: $\DD, \DD^\perp, \Drk, \Drk^\perp$.

The time and memory needed to decompose $\Drk^\perp$ represents an approximate lower bound on the cost of computing cycle representatives in a sequential non-greedy fashion\footnote{By contrast,  the method we propose is sequential and greedy.}.  Indeed, this is essentially the approach taken in \cite{vcufar2020ripserer}, where pivot-elements are precomputed via the cohomology algorithm, and $\DD_{\upto{m} \pcols}^\perp$ is decomposed by a procedure nearly identical to Algorithm \ref{alg:revised_lrdec}.  Since we include the decomposition of $\Drk^\perp$ in our analysis, we also include $\Drk$, for symmetry.

\subsection{Data sets}

Data used for the experiments are described below; all data sets and code for generating simulated data are available at \cite{UMatch_code}.

\smallskip
\noindent \textbf{Gaussian Random Fields} (GRF2DAni, GRF2DExp, GRF3DAni, and GRF3DExp): Gaussian random fields with exponential and anisotropic covariance structure in dimensions two and three were generated via the Julia package \texttt{GaussianRandomFields.jl}.  These fields were formatted as $(1000 \times 1000)$ and $(50 \times 50 \times 50)$ pixel arrays, then converted to filtered cubical complexes via the ``T-construction'' \cite{garin2020duality}.
\smallskip

\noindent \textbf{Erdos-Renyi} (ER100 and ER150): Filtered clique complexes for complete edge-weighted graphs on $100$ and $150$ vertices with weights drawn iid from the uniform distribution  
\smallskip

\noindent  \textbf{Uniform} (Uniform): A Vietoris-Rips complex for $500$ points sampled uniformly from the unit cube in $\mathbb{R}^{20}$ under the standard Euclidean metric.
\smallskip

\noindent \textbf{Torus} (Torus): A Vietoris-Rips complex for $500$ points sampled uniformly from the unit cube in $\mathbb{R}^3$, equipped with the metric $d(x,y) = \min_{z} \vert\vert x-(y+z)\vert\vert_2$, where $z$ runs over all points in $\{0, 1, -1\}^3$.
\smallskip

\noindent \textbf{Henneberg} (Henne): A Vietoris-Rips complex for $1000$ points subsampled randomly from the $5456$ points on the Henneberg surface inthree dimensions available at \cite{Stolz2020}. 

\noindent \textbf{Cyclo-octane} (Cyclo):  A Vietoris-Rips complex for $1000$ points subsampled randomly from the $6040$ points in $\mathbb{R}^{24}$ from the Cyclo-octane data set  available at \cite{Stolz2020}.

\begin{remark}
Very significant theoretical progress has been made in the theory of random cell complexes \cite{kahle2011random, linial2006homological, costa2016random, meshulam2009homological, adler2010persistent, hiraoka2018limit}.  However, since data sets used in applications rarely conform to the assumptions of any one probability model, it is convention to evaluate the performance of decomposition algorithms via benchmarking on a range of scientific data sets.  
\end{remark}

\subsection{Software}

All software is implemented in the Rust programming language; code for running these experiments is available at \cite{UMatch_code}. Results as described below are for computations  performed on a Dell PowerEdge R730xd server with 2.4GHz Intel Xeon E5-2640 processors and 512 GB of RAM, running Ubuntu 18.04 LTS.

The implementation of Algorithm \ref{alg:revised_lrdec} stores upper triangular matrices in CSR format, and matching matrices as a hash map.  For both clique and cubical complexes, we store the matrix $\DD$ in a compressed data structure that permits lazy look-up of rows and columns. While doing row reductions, the rows of $\DD$ are generated by a coface iterator; correspondingly, while doing column reductions, the columns of $\DD$ are generated by a face iterator.  The data structure for clique complexes closely resembles that of   \cite{bauer2019ripser}, and the structure used for cubical complexes closely resembles that of \cite{maria2014gudhi, kaji2020cubical}.  

Unlike CSC/CSR storage formats, the lazy structure that encodes $\DD$ has no analog of a ``transpose'' operation that can significantly speed up or slow down read-access to rows or columns.  Nor does it have a natural ``sub-index'' operation (since each row/column is built on the fly).  Therefore, each variant on   $\DD^\perp, \Drk, \Drk^\perp$, is encoded by a wrapper object which translates a call for a specific row or column into a call for a different row or column, as necessary.  To decompose $\DD, \DD^\perp, \Drk$, or $ \Drk^\perp$, we first load the source data for $\DD$ and the index matching relation $\IMatch$.  The necessary wrapper object is then constructed from $\DD$ and $\IMatch$.  

\begin{remark}
The statistics reported in Tables \ref{table:time} and \ref{table:memory} include the time required to load both $\DD$ \underline{and} the set of all matched row/column indices, even when $\IMatch$ is unused because the matrix to be decomposed is $\DD$ or $\DD^\perp$.  This convention is used in order to avoid variable confounds in benchmark results.
\end{remark}

In the case of clique complexes, we also implement the minimal/steepness/emergent optimization (\S \ref{sec:shortcircuitph}) to short-circuit construction of certain rows and columns.  This strategy makes less sense for cubical data, both because it is harder to find the leftmost element of a row (or the rightmost element of a column) in this regime, and because each row and column has a small number of nonzero entries in general (6 at most), which places a low ceiling on the potential benefit of short-circuiting the construction of these objects.

\subsection{Results and discussion}
We report the results of the experiments described above in Tables \ref{table:entries}, \ref{table:time}, and \ref{table:memory} at the end of the document. For each experiment, Table  \ref{table:entries} gives the number of non-zero entries in the different matrices we consider,  Table \ref{table:time} gives the decomposition time, and Table \ref{table:memory} gives the peak heap memory used during the decomposition memory. 

Our experiments verify postulate \ref{item:matrixassumptioncompressed}  that $\Rirr$ consistently uses less memory than $\Ri$; see Table \ref{table:entries}. This finding is commensurate with results from \cite{zhang2019hypha, de2011dualities} and other sources, which  suggest that the number of algebraic operations needed to reduce non-pivot rows and columns to zero tends to far outstrip the number of operations needed to reduce pivot rows and columns; note, in particular, that there is a 1-1 correspondence between row additions and off-diagonal entries of $\Ri$.  This difference exists \emph{even} for cubical complexes, where the gap in time and memory to reduce $\DD$ versus $\DD^\perp$ disappears.

Regarding postulate \ref{item:matrixassumptionantitranspose}, we find the decomposition time is similar for $\DD$ and $\DD^\perp$ when $\DD$ comes from a cubical complex and substantially faster for $\DD$ than $\DD^\perp$ when $\DD$ comes from a clique complex. These results recapitulate existing findings in the literature. In particular, it was noted in \cite{bauer2014distributed} that the gap in time/memory to reduce $\DD$ versus $\DD^\perp$ disappears for cubical complexes.

 Postulates \ref{item:matrixassumptioncompressed} and \ref{item:matrixassumptionantitranspose} jointly support the decision to place $\Rirr$ at the heart of our lazy decomposition, storage, and retrieval scheme.  In particular, decomposition of $\DD$ (which exposes $\Ri$) is faster than decomposition of $\DD^\perp$ (which exposes $\CC$), and storage of $\Rirr$ costs considerably less than storage of $\Ri$.
 
 \medskip

For postulate \ref{item:matrixassumptionlazy}, we find that clearing a pivot block by column operations uses similar-order-of-magnitude time and memory as does clearing the same pivot block by row operations.  In the standard approach to reducing a filtered boundary operator by row-reduction, one would apply the cohomology algorithm to $\DD_{\prows \upto{n}},$ having deleted all nonpivot rows via the clearing/compression optimization\footnote{Compression does not remove \underline{all} non-pivot rows, in general.  However, since the final space in our filtration is nullhomotopic, the image of $\partial_2$ equals the kernel of $\partial_1$, hence (by the rank-nullity theorem) the number of pivot columns of $\partial_1$ equals the number of non-pivot rows of $\partial_2$; therefore, compression removes all non-pivot rows for each of the data sets studied here.}. In this case, the time to reduce $\DD_{\prows \upto{n}}$ would be roughly similar to that of applying the same algorithm to $\Drk$ since construction of the vast majority of pivot rows is short-circuited, c.f.\ \cite{bauer2019ripser}.  Thus, the time needed to compute generators by first row-reducing (thus, revealing the pivot elements) then column-reducing the submatrix indexed by pivot indices, is approximately double the time required only to row-reduce.  On the other hand, performing a sparse triangular solve requires only a fraction of a second.  Thus, in regimes where only a small number of generators are required, the lazy approach offers concrete performance advantages.
    
The time and memory required to reduce $\Drk$ roughly approximates that required to reduce $\DD_{\prows \upto{n}}$, which is the matrix one would reduce with the standard cohomology algorithm under the classical clearing optimization.  Thus, our results replicate the finding that compression  accelerates computation (Table \ref{table:time}), even when the number of non-pivot rows or columns is small.

\begin{table}[ht]
\centering
\begin{tabularx}{\textwidth}{YYYYY}
  \toprule
        &&\multicolumn{3}{c}{number of nonzero entries}\\
         \cmidrule{3-5}
       dataset&size of $\partial_2$&$\MM$&$\Ri - \id$&$\Rirr - \id$\\
        \midrule
        GRF2DAni&$2,002,000\times1,000,000$&$1,000,000$&$70,471,346$&$15,634$\\
        \midrule
        GRF3DAni&$390,150\times382,500$&$257,500$&$12,043,279$&$163,733$\\
        \midrule
        GRF2DExp&$2,002,000\times1,000,000$&$1,000,000$&$53,229,639$&$310,373$\\
        \midrule
        GRF3DExp&$390,150\times382,500$&$257,500$&$9,270,009$&$77,735$\\
        \midrule
        ER100&$4,657\times134,654$&$4,558$&$49,561$&$4,241$\\
        \midrule
        ER150&$10,846\times504,017$&$10,697$&$154,897$&$17,245$\\
        \midrule
        Uniform&$112,586\times15,586,723$&$112,087$&$1,693,846$&$5,019$\\
        \midrule
        Torus&$91,162\times9,314,575$&$90,663$&$2,486,139$&$1,211$\\
        \midrule
        Henne&$411,484\times100,278,849$&$410,485$&$23,885,893$&$10,126$\\
        \midrule
        Cyclo&$300,712\times 47,272,174$&$299,713$&$14,897,198$&$5,572$\\
        \bottomrule
\end{tabularx}
\caption{Nonzero entries for several (sub)matrices associated with  U-match decomposition.  Each row corresponds to a clique or cubical complex, $X$. We pass the dimension-2 boundary operator $ \partial_2: \Chains_2(X, \binary) \to \Chains_1(X, \binary)$, represented as a matrix $\DD$, to Algorithm \ref{alg_lrdec} in order to obtain a U-match decomposition $\RR \MM = \DD \CC$ (equivalently, $\Ri \DD = \MM \Ci$). 
The two righthand columns report number of \underline{\smash{off-diagonal}} nonzero entries for $\Ri$ and $\Rirr$, respectively.  
Recall that matrix $\MM$ is uniquely determined by $\DD$; the number of nonzero entries in this matrix equals the number of pivot elements of the decomposition, and this number does not depend on choice of decomposition algorithm. By contrast, matrix $\Ri$ is not uniquely determined by $\DD$, and other decomposition algorithms may yield different results.  In this experiment, the number of off-diagonal entries in $\Rirr$ is typically even smaller than the number of nonzero entries in $\MM$, often by several orders of magnitude.
}
\label{table:entries}
\end{table}

\begin{table}[h]
\centering
\begin{tabularx}{\textwidth}{YYYYY}
  \toprule
        &\multicolumn{4}{c}{decomposition timing (seconds)}\\
         \cmidrule{2-5}
         &\multicolumn{2}{c}{row clearing}&\multicolumn{2}{c}{column clearing}\\
       dataset&full matrix&pivot block&full matrix&pivot block\\
        \midrule
        GRF2DAni&$128.80$&$6.12$&$8.36$&$9.25$\\
        \midrule
        GRF3DAni&$60.64$&$2.18$&$43.09$&$2.21$\\
        \midrule
        GRF2DExp&$85.03$&$7.06$&$9.43$&$9.37$\\
        \midrule
        GRF3DExp&$35.64$&$1.87$&$28.45$&$1.97$\\
        \midrule
        ER100&$4.77$&$0.20$&$98.00$&$0.07$\\
        \midrule
        ER150&$46.84$&$1.52$&$1,649.75$&$0.39$\\
        \midrule
        Uniform&$1,086.58$&$20.72$&$14,917.19$&$13.89$\\
        \midrule
        Torus&$1,042.05$&$11.79$&$973.93$&$8.02$\\
        \midrule
        Henne&$29,712.27$&$169.91$&$30,563.55$&$128.29$\\
        \midrule
        Cyclo&$11,688.37$&$76.77$&$47,381.64$&$48.65$\\
        \bottomrule
\end{tabularx}
\caption{Execution time for Algorithm \ref{alg:revised_lrdec}.  Each row corresponds to the dimension-2 boundary operator $ \partial_2: \Chains_2(X, \binary) \to \Chains_1(X, \binary)$ of a clique or cubical complex, $X$.  We identify this operator with its matrix representation, $\DD$. For each $X$, we pass $\DD, \DD^\perp, \Drk$ and $\Drk^\perp$ to Algorithm  \ref{alg:revised_lrdec} to obtain a compressed representation of a U-match decomposition $\RR \MM = \DD \CC$.  Reduction time for pivot blocks are roughly similar for row vs.\ column operations in both clique and cubical complexes.  For full matrices, column reduction is sometimes faster for cubical complexes, and sometimes slower for clique complexes.
}
\label{table:time}
\end{table}

\begin{table}[h]
\centering
\begin{tabularx}{\textwidth}{YYYYY}
  \toprule
        &\multicolumn{4}{c}{decomposition peak heap memory use (kb)}\\
         \cmidrule{2-5}
         &\multicolumn{2}{c}{row clearing}&\multicolumn{2}{c}{column clearing}\\
       dataset&full matrix&pivot block&full matrix&pivot block\\
        \midrule
        GRF2DAni&$1,241,720$&$1,0836,88$&$1,031,244$&$1,094,752$\\
        \midrule
        GRF3DAni&$321,644$&$279,016$&$321,348$&$279,300$\\
        \midrule
        GRF2DExp&$1,244,532$&$1,088,636$&$1,038,472$&$1,093,644$\\
        \midrule
        GRF3DExp&$321,348$&$278,900$&$321,780$&$276,096$\\
        \midrule
        ER100&$30,612$&$16,208$&$32,580$&$16,056$\\
        \midrule
        ER150&$109,296$&$45,788$&$98,412$&$45,768$\\
        \midrule
        Uniform&$2,002,032$&$1,231,860$&$1,881,004$&$1,231,844$\\
        \midrule
        Torus&$1,121,316$&$743,848$&$1,241,508$&$743,904$\\
        \midrule
        Henne&$13,449,404$&$7,808,640$&$10,132,272$&$7,808,800$\\
        \midrule
        Cyclo&$6,663,060$&$3,719,136$&$5,254,084$&$3,719,108$\\
        \bottomrule
\end{tabularx}
\caption{Memory use (peak heap) for Algorithm \ref{alg:revised_lrdec}.  Each row corresponds to the dimension-2 boundary operator $ \partial_2: \Chains_2(X, \binary) \to \Chains_1(X, \binary)$ of a clique or cubical complex, $X$.  We identify this operator with its matrix representation, $\DD$. For each $X$, we pass $\DD, \DD^\perp, \Drk$ and $\Drk^\perp$ to Algorithm  \ref{alg:revised_lrdec} to obtain a compressed representation of a U-match decomposition $\RR \MM = \DD \CC$.  Memory use is roughly similar for reduction of $\Drk$ versus $\Drk^\perp$ (that is, decomposition of the pivot matrix by row versus column operations).  Memory use is also similar for reduction of $\DD$ versus $\DD^\perp$ (that is, decomposition of $\DD$ by row versus column operations).  
}
\label{table:memory}
\end{table}

\section{Conclusion}
\label{sec:conclusion}
A host of problems in modern TDA can be answered with homological algebra, and indeed, linear algebra.  However, data structures with the capacity to store large quantities of linear data have proven to be a decisive bottleneck.  The absence of such structures limits our ability to perform an essential task in algebraic topology -- namely, to extract knowledge about shapes from diagrams of maps.  

The U-match strategy addresses this problem through compression and lazy evaluation.  It flexibly adapts to diverse problems in applied homological algebra, in particular computing bases for subspaces of cycles and boundaries. 
At the time of this writing, cycle representatives can be computed by a handful of TDA software packages but remain difficult to analyze or work with because the software that computes them presents the results \emph{a la carte} without exposed access to the underlying chain bases and matrices. 
Through U-match, we gain  access not only to cycle representatives but also to the bases and matrices necessary to manipulate them.   Our experiments, detailed in \S\ref{sec:experiments}, demonstrate the computational efficacy of this approach.
This same framework makes accessible the computations needed to implement an array of techniques from algebraic topology beyond simple persistent homology computations, at a scale that allows us to work with real data. Computing induced maps on (persistent) homology and working with diagrams and exact sequences are fundamental in pure topology. While the authors defer our own efforts in this direction to forthcoming and future work, we believe that access to these same methods will be of great utility to applied topologists in general.

\medskip

In developing the U-match framework, several directions for future effort became apparent. As we have seen, computation time and resources are always a constraint. Thus, developing variants of these methods that are appropriate for implementation in distributed computing environments or via GPU acceleration would be of value in many contexts. In addition, case-by case study of each of the 64 calculations described in Remark \ref{rmk:64cases} would potentially provide insights and further refinements of these methods in particular cases. Finally, we have performed only the most obvious experiments with these tools; further statistical analysis of the various components of the U-match decomposition in different contexts would improve our understanding of its capabilities and limitations.
	
\section*{Acknowledgements}

The authors would like to thank David Turner and Bryn Keller for their many contributions to the project. This material is based upon work supported by the National Science Foundation under grants  DMS-1854683, DMS-1854703 and DMS-1854748.

\bibliography{ffwg_bib}

\appendix

\section{Block identities}
\label{sec:blockidentities}

Posit a proper U-match decomposition 
    \begin{align}
    \RR \MM = \DD \CC.
    \tag{\ref{eq:umatchdef}}
    \end{align}
We showed in \S\ref{sec:overallblockstructure} that
    \begin{align}
    \begin{array}{l |cc|}
        \multicolumn{1}{c}{}& \multicolumn{1}{c}{\nprows} & \multicolumn{1}{c}{\prows}  \\  \cline{2-3}
        \nprows & \id & \RR_{\nprows\prows}   \\   
        \prows &  &  \RR_{\prows  \prows} \\ \cline{2-3}     
    \end{array}     
    &&    
    \begin{array}{l |cc|}
        \multicolumn{1}{c}{}& \multicolumn{1}{c}{\pcols} & \multicolumn{1}{c}{\npcols}  \\  \cline{2-3}
        \nprows &  &    \\   
        \prows & \MM_{\prows  \pcols} &   \\ \cline{2-3}     
    \end{array}     
    &&
    =
    &&
    \begin{array}{l |cc|}
        \multicolumn{1}{c}{}& \multicolumn{1}{c}{\pcols} & \multicolumn{1}{c}{\npcols}  \\  \cline{2-3}
        \nprows & \DD_{\nprows  \pcols} & \DD_{\nprows  \npcols}   \\   
        \prows & \DD_{\prows  \pcols} &  \DD_{\prows  \npcols} \\ \cline{2-3}     
    \end{array}     
    &&
    \begin{array}{l |cc|}
        \multicolumn{1}{c}{}& \multicolumn{1}{c}{\pcols} & \multicolumn{1}{c}{\npcols}  \\  \cline{2-3}
        \pcols & \CC_{\pcols\pcols} & \CC_{\pcols\npcols}   \\   
        \npcols &  &  \id \\ \cline{2-3}     
    \end{array} 
    \tag{\ref{eq:umatchblocks}}
	\end{align}  
In particular, axioms \ref{item:proper1} and \ref{item:proper2}, which define what it means for a U-match decomposition to be proper, are equivalent to equations 
    \begin{align}
        \CC_{\npcols \upto{n}} &= \id_{\npcols \upto{n}}
        \label{eq:properaxiom_blockform_ColOper}   
        \\
        \RR_{\upto{m} \nprows} &= \id_{\upto{m} \nprows }
        \label{eq:properaxiom_blockform_RowOper}
    \end{align}
respectively.

If we write $\prows^\perp = (\prows_k, \ldots, \prows_1)$ for the reverse of a finite sequence $\prows$, then the anti-transposed U-match decomposition $(\Ci)^\perp \MM = \DD^\perp (\Ri)^\perp$ has an analogous block structure:
    \begin{multline}
        \begin{array}{l |cc|}
        \multicolumn{1}{c}{}& \multicolumn{1}{c}{\npcols^\perp} & \multicolumn{1}{c}{\pcols^\perp}  \\  \cline{2-3}
        \npcols^\perp & \id & (\Ci_{\pcols\npcols})^\perp   \\   
        \pcols^\perp &  &  (\Ci_{\pcols\pcols})^\perp \\ \cline{2-3}     
    \end{array} 
    \quad 
    \begin{array}{l |cc|}
        \multicolumn{1}{c}{}& \multicolumn{1}{c}{\prows^\perp} & \multicolumn{1}{c}{\nprows^\perp}  \\  \cline{2-3}
        \npcols^\perp &  &    \\   
        \pcols^\perp & \MM_{\prows  \pcols}^\perp &   \\ \cline{2-3}     
    \end{array} 
     \\
    =
    \quad
    \begin{array}{l |cc|}
        \multicolumn{1}{c}{}& \multicolumn{1}{c}{\prows^\perp} & \multicolumn{1}{c}{\nprows^\perp}  \\  \cline{2-3}
        \npcols^\perp & \DD_{\nprows  \pcols}^\perp & \DD_{\nprows  \npcols}^\perp   \\   
        \pcols^\perp & \DD_{\prows  \pcols}^\perp &  \DD_{\prows  \npcols}^\perp \\ \cline{2-3}     
    \end{array}     
    \quad
    \begin{array}{l |cc|}
        \multicolumn{1}{c}{}& \multicolumn{1}{c}{\prows^\perp} & \multicolumn{1}{c}{\nprows^\perp}  \\  \cline{2-3}
        \prows^\perp & (\Ri_{\prows  \prows})^\perp & (\Ri_{\nprows\prows})^\perp   \\   
        \nprows^\perp &  & \id   \\ \cline{2-3}     
    \end{array} 
    \label{eq:umatchblocks_antitranspose}
	\end{multline}  
	
\begin{remark}
\label{rmk:pivotnonpivotorder}
    Notice, in particular, that the symbol representing the sequence of non-pivot column indices, $\npcols^\perp$, appears above and to the left of the symbol for pivot columns indices $\pcols^\perp$  wherever these symbols appear as row/column labels in Equation \eqref{eq:umatchblocks_antitranspose}.  This ordering is reversed in Equation \eqref{eq:umatchblocks}.  A similar observation holds for row labels. 
\end{remark}

We claim that 
    \begin{alignat}{4}
        \MM \Ci
        \;
        &= 
        \;
        \Ri \DD
        \;
        &&
        \equiv 
        \begin{array}{l |c|l}
            \multicolumn{1}{c}{}
            &
            \multicolumn{1}{c}{\upto{n}} 
            &  \multicolumn{1}{c}{} \\  \cline{2-2}
            \nprows &    \\ 
            \cline{2-2}
            \prows & 
            \Rirr \DD_{\prows  \upto{n}} &   \\ \cline{2-2}    
        \end{array}          
        \label{eq:umatchrowidentity}
    \\
        \RR \MM 
        \;
        &= 
        \;    
        \DD \CC
        \;    
        &&
        \equiv 
        \begin{array}{l |c|c|}
            \multicolumn{1}{c}{}& \multicolumn{1}{c}{\pcols} & \multicolumn{1}{c}{\npcols}  \\  \cline{2-3}
            \multirow{2}{*}{ $\upto{m}$ } &
            \multirow{2}{*}{ $\DD_{\upto{m}  \pcols} \CC_{\pcols \pcols}$ }
             &
            \\   
             &   &   \\ \cline{2-3}     
        \end{array}      
        \label{eq:umatchcolidentity} 
    \end{alignat}    
For proof, let us focus first on Equation \eqref{eq:umatchrowidentity}.  Identity $\MM \Ci = \Ri \DD$ follows directly from the defining equation \eqref{eq:umatchdef}.  Equation \eqref{eq:umatchblocks}  implies that $\Ri \DD$ has the following block structure 
\begin{align*}
     \Ri \DD
    \equiv 
    \begin{array}{l |cc|}
        \multicolumn{1}{c}{}& \multicolumn{1}{c}{\nprows} & \multicolumn{1}{c}{\prows}  \\  \cline{2-3}
        \nprows & \id &  (\Ri)_{\nprows \prows}    \\   
        \prows &  &  (\Ri)_{\prows \prows} \\ \cline{2-3}     
    \end{array}  
    \;
    \;
    \begin{array}{|c|}
        \multicolumn{1}{c}{\upto{n}} \\  \cline{1-1}
        \DD_{\nprows  \upto{n}}       \\   
         \DD_{\prows  \upto{n}}   \\ \cline{1-1}     
    \end{array} 
\end{align*}
hence $(\Ri \DD)_{\prows \upto{n}} = \Rirr \DD_{\prows \upto{n}}$; this accounts for the lower half of the matrix on the righthand side of \eqref{eq:umatchrowidentity}.  On the other hand, every non-pivot row of $\MM$ vanishes, hence $(\MM \Ci)_{\nprows \upto{n}}=0 $.  This accounts for the upper half of the matrix on the righthand side of \eqref{eq:umatchrowidentity}, and completes the proof of that equation.  Equation \eqref{eq:umatchcolidentity} can be argued in a similar fashion, invoking the identity 
\begin{align*}
    \DD \CC
    \equiv 
    \begin{array}{l |c|c|}
        \multicolumn{1}{c}{}& \multicolumn{1}{c}{\pcols} & \multicolumn{1}{c}{\npcols}  \\  \cline{2-3}
        \multirow{2}{*}{ $\upto{m}$ } &
        \multirow{2}{*}{ $\DD_{\upto{m}  \pcols}$ }
         &  \multirow{2}{*}{ $\DD_{\upto{m}  \npcols}$ }   \\   
         &   &   \\ \cline{2-3}     
    \end{array}     
    \;
    \;
    \begin{array}{|c|c|l}
        \multicolumn{1}{c}{\pcols} & \multicolumn{1}{c}{\npcols} &  \\  \cline{1-2}
        \CC_{\npcols\pcols} & \CC_{\pcols \npcols }&   \multicolumn{1}{c}{ \pcols}  \\   
        & \id & \multicolumn{1}{c}{ \npcols} \\ \cline{1-2}     
    \end{array}    
\end{align*}
Alternatively, one could prove Equation \eqref{eq:umatchcolidentity}  by showing that it is equivalent to Equation \eqref{eq:umatchrowidentity}, via anti-transpose duality.

Since the non-pivot rows of $\CC$ and the nonpivot columns of $\RR$ (hence also of $\Ri$) are  unit vectors (this holds by Axioms \ref{item:proper1} and \ref{item:proper1}, equivalently, by Equations \eqref{eq:properaxiom_blockform_ColOper} and \eqref{eq:properaxiom_blockform_RowOper}), it follows from Equations \eqref{eq:umatchrowidentity} and \eqref{eq:umatchcolidentity} that 
    \begin{align}
        \Ci
        \;
        &
        \equiv 
        \;
        \begin{array}{l |c|l}
            \multicolumn{1}{c}{}
            &
            \multicolumn{1}{c}{\upto{n}} 
            &  \multicolumn{1}{c}{} \\  \cline{2-2}
            \pcols & \Mirk \Rirr \DD_{\prows \upto{n}}    \\ 
            \cline{2-2}
            \npcols & \id_{\npcols \upto{n}} 
             &   \\ \cline{2-2}    
        \end{array}          
        \label{eq:umatchrowidentityInner}
    \\
        \RR 
        \;    
        &
        \equiv 
        \begin{array}{l |c|c|}
            \multicolumn{1}{c}{}& \multicolumn{1}{c}{\nprows} & \multicolumn{1}{c}{\prows}  \\  \cline{2-3}
            \multirow{2}{*}{ $\upto{m}$ } &
            \multirow{2}{*}{ $\id_{\upto{m} \nprows}$ }
             &
            \multirow{2}{*}{ $\DD_{\upto{m}  \pcols} \CC_{\pcols \pcols} \Mirk$ }
            \\   
             &   &   \\ \cline{2-3}     
        \end{array}      
        \label{eq:umatchcolidentityInner} 
    \end{align}

\subsection{Inner identities and proof of Theorem \ref{thm:umatchblockidentities}} 
\label{sec:proofofmatrixidentities}
    
Let us define
    \begin{align}
        \AA = \Rirr \Drk
    \end{align}
Then 
    \begin{align}
        \Drk &= \Rrr \AA 
        \\
        \Drki &= \Ai \Rirr
        \label{eq:smallmatrixconversion}
    \end{align}
It follows from Equation \eqref{eq:umatchblocks} that $\Rrr \Mrk  = \Drk \Ccc$, therefore
    \begin{align}
        \Mrk    & = \Rirr \Drk \Ccc \\
                & = \AA \Ccc\\
    \end{align}
Hence
    \begin{align}
        \Ccc & = \Ai \Mrk
        \label{eq:ARtip}
    \end{align}

We are now ready to prove Theorem \ref{thm:umatchblockidentities}.  Let us first recall the statement of this result:

\newtheorem*{thm:associativity}{Theorem \ref{thm:umatchblockidentities}}
\begin{thm:associativity}[Inner identities]
Posit a proper U-match decomposition $\RR \MM = \DD \CC$, and let $\AA = \Rirr \DD_{\prows  \pcols} = (\Ri \DD)_{\prows \pcols}$.  Then the  following matrix identities hold, where
    \begin{enumerate}
        \item blank entries indicate zero blocks
        \item permutations on the rows and columns of $\RR, \Ri, \CC, \Ci, \Ri \DD,$ and $\DD \CC$ are indicated by the sequences $\prows, \nprows, \pcols, \npcols, \upto{m},\upto{n}$, which appear as labels for blocks of row and column indices.
    \end{enumerate}
\begin{alignat}{2}
    &
    \CC 
    \;
    &&\equiv 
    \;
    \begin{array}{l |cc|}
        \multicolumn{1}{c}{}& \multicolumn{1}{c}{\pcols} & \multicolumn{1}{c}{\npcols}  \\  \cline{2-3}
        \pcols & \Ai \MM_{\prows \pcols} & - \Ai \Rirr \DD_{\prows \npcols}     
        \\   
        \npcols &  &  \id \\ \cline{2-3}     
    \end{array}     
    \tag{\ref{eq:inneridentity1}}    
    \\[1em]
    %
    &
    \Ci
    \;
    &&\equiv 
    \;
    \begin{array}{l |c|l}
        \multicolumn{1}{c}{}
        &
        \multicolumn{1}{c}{\upto{n}} 
        &  \multicolumn{1}{c}{} \\  \cline{2-2}
        \pcols & \MM_{\prows \pcols}^{-1} \Rirr \DD_{\prows  \upto{n}}   \\ 
        \cline{2-2}
        \npcols & \id_{\npcols \upto{n} }
        &   \\ \cline{2-2}    
    \end{array}   
    \tag{\ref{eq:inneridentity2}}    
    \\[1em]
    %
    &
    \Ri 
    \;
    &&\equiv 
    \;
    \begin{array}{l |cc|}
        \multicolumn{1}{c}{}& \multicolumn{1}{c}{\nprows} & \multicolumn{1}{c}{\prows}  \\  \cline{2-3}
        \nprows & \id &  -\DD_{\nprows  \pcols} \DD_{\prows \pcols}^{-1}   \\
        \prows &  &  \Rirr \\ \cline{2-3}     
    \end{array}    
    \tag{\ref{eq:inneridentity3}}    
    \\[1em]
    %
    &
    \RR
    \;
    &&\equiv 
    \;
    \begin{array}{l |c|c|}
        \multicolumn{1}{c}{}& \multicolumn{1}{c}{\nprows} & \multicolumn{1}{c}{\prows}  
        \\  \cline{2-3}
         \upto{m} 
        &
        \id_{\upto{m} \nprows} 
         &
         \DD_{\upto{m}  \pcols} \Ai
         \\ \cline{2-3}     
    \end{array}     
    \tag{\ref{eq:inneridentity4}}
\end{alignat}
\end{thm:associativity}
\begin{proof}
Equation \eqref{eq:ARtip} provides the last equality in the following sequence
    \begin{align*}
        \RR_{\upto{m} \prows} 
        = 
        \DD \CC_{\upto{n} \pcols} \Mirk 
        =
        \DD_{\upto{m} \pcols} \CC_{ \pcols \pcols} \Mirk         
        =
        \DD_{\upto{m} \pcols } \Ai
    \end{align*}
Identity \eqref{eq:inneridentity4} follows, if we recall that non-pivot columns are unit vectors, as per Equation \eqref{eq:properaxiom_blockform_RowOper}.  It follows, therefore, that 
    \begin{align}
        \RR
        \equiv
        \begin{array}{l |cc|}
            \multicolumn{1}{c}{}& \multicolumn{1}{c}{\nprows} & \multicolumn{1}{c}{\prows}  \\  \cline{2-3}
            \nprows & \id &  \DD_{\nprows  \pcols} \Ai   \\
            \prows &  &  \Rrr \\ \cline{2-3}     
        \end{array}         
        \label{eq:rrblockformwithA}
    \end{align}
One can verify that the product of the two matrices on the righthand sides of Equations \eqref{eq:rrblockformwithA} and \eqref{eq:inneridentity3} is equal to $I$.  This proves Equation \eqref{eq:inneridentity3}, since inverses are unique.

Equation \eqref{eq:inneridentity2} was proved in the preceding discussion (Equation \eqref{eq:properaxiom_blockform_ColOper}).  If we assume Equation \eqref{eq:inneridentity1}, then $(\CC \Ci)_{\pcols \pcols}  =  \Ai \Mrk \cdot \Mirk \Rirr \Drk = \Drki \Drk = \id$ and Equation \eqref{eq:smallmatrixconversion} yields both under braces in the following expression:
    \begin{align*}
        (\CC \Ci)_{\pcols \npcols} 
        =
        \underbrace
            {\Ai \Mrk \Mirk \Rirr} 
            _
            {\Drki}
            \DD_{\prows \npcols}
        -
        \underbrace
            {\Ai \Rirr} 
            _
            {\Drki}     
            \DD_{\prows \npcols}
        =
        0
    \end{align*}
In particular, the formula for $\Ci$ given in Equation \eqref{eq:inneridentity1} satisfies the condition $\Ci \CC = \id$.  Correctness of Equation \eqref{eq:inneridentity1} follows by uniquenss of inverses.  This completes the proof.
\end{proof}

\subsection{Further identities}

\begin{proposition}
\label{prop:umatchblockidentities}
Let $\AA = \Rirr \DD_{\prows  \pcols} = (\Ri \DD)_{\prows \pcols}$.  The following matrix identities hold, where
    \begin{enumerate}
        \item blank entries indicate zero blocks
        \item by abuse of notation, the matrices $\RR, \Ri, \CC, \Ci, \Ri \DD,$ and $\DD \CC$ that appear on the left side of each equation are understood to have their rows and columns permuted as indicated by the row/column labels on the right side of each equation.
    \end{enumerate}
\begin{align*}
    \CC 
    \;
    &= 
    \;
    \begin{array}{l |cc|}
        \multicolumn{1}{c}{}& \multicolumn{1}{c}{\prows} & \multicolumn{1}{c}{\npcols}  \\  \cline{2-3}
        \pcols & \AA^{-1} &    \\   
        \npcols &  &  \id \\ \cline{2-3}     
    \end{array}     
    \;
    \;
    \begin{array}{|cc|}
        \multicolumn{1}{c}{\prows} & \multicolumn{1}{c}{\npcols}  \\  \cline{1-2}
        \id & -\Rirr \DD_{\prows  \npcols}   \\   
        &  \id \\ \cline{1-2}     
    \end{array}     
    \;
    \;
    \begin{array}{|cc|l}
        \multicolumn{1}{c}{\pcols} & \multicolumn{1}{c}{\npcols} & \multicolumn{1}{c}{} \\  \cline{1-2}
        \MM_{\prows  \pcols} & & \    \\   
         &  \id &  \\ \cline{1-2}     
    \end{array}    
     && \CC_{\pcols\npcols} & = - (\DD_{\prows  \pcols})^{-1} \DD_{\prows  \npcols}
    \\[1em]
    %
    \Ci
    \;
    &\equiv 
    \;
    \begin{array}{l |cc|}
        \multicolumn{1}{c}{}& \multicolumn{1}{c}{\prows} & \multicolumn{1}{c}{\npcols}  \\  \cline{2-3}
        \pcols & (\MM_{\prows \pcols})^{-1} &    \\   
        \npcols &  &  \id \\ \cline{2-3}     
    \end{array}     
    \;
    \;
    \begin{array}{|cc|}
        \multicolumn{1}{c}{\prows} & \multicolumn{1}{c}{\npcols}  \\  \cline{1-2}
        \id & \Rirr \DD_{\prows  \npcols}   \\   
        &  \id \\ \cline{1-2}     
    \end{array}     
    \;
    \;
    \begin{array}{|cc|l}
        \multicolumn{1}{c}{\pcols} & \multicolumn{1}{c}{\npcols} & \multicolumn{1}{c}{} \\  \cline{1-2}
        \AA & &     \\   
         &  \id &  \\ \cline{1-2}     
    \end{array}   
    \\[1em]
    %
    \Ri 
    \;
    &\equiv 
    \;
    \begin{array}{l |cc|}
        \multicolumn{1}{c}{}& \multicolumn{1}{c}{\nprows} & \multicolumn{1}{c}{\prows}  \\  \cline{2-3}
        \nprows & \id &  -\DD_{\nprows  \pcols} \AA^{-1}   \\   
        \prows &  &  \id \\ \cline{2-3}     
    \end{array}     
    \;
    \;
    \begin{array}{|cc|l}
        \multicolumn{1}{c}{\nprows} & \multicolumn{1}{c}{\prows} & \multicolumn{1}{c}{} \\  \cline{1-2}
        \id & &     \\   
         &  \Rirr &  \\ \cline{1-2}     
    \end{array}  
    &&
    (\Ri)_{\nprows\prows} &= -\DD_{\nprows  \pcols} (\DD_{\prows  \pcols})^{-1}
    \\[1em]
    %
    \RR
    \;
    &\equiv 
    \;
    \begin{array}{l |cc|}
        \multicolumn{1}{c}{}& \multicolumn{1}{c}{\nprows} & \multicolumn{1}{c}{\prows}  \\  \cline{2-3}
        \nprows & \id &     \\   
        \prows &  &  \RR_{\prows  \prows} \\ \cline{2-3}     
    \end{array}  
    \;
    \;
    \begin{array}{|cc|l}
        \multicolumn{1}{c}{\nprows} & \multicolumn{1}{c}{\prows} & \multicolumn{1}{c}{} \\  \cline{1-2}
        \id & \DD_{\nprows  \pcols} \AA^{-1} &     \\   
         &  \id &  \\ \cline{1-2}     
    \end{array}      
\end{align*}

Moreover,
\begin{align*}
    \MM \Ci
    \;
    &= 
    \;
    \Ri \DD
    \;
    \equiv 
    \begin{array}{l |c|l}
        \multicolumn{1}{c}{}
        &
        \multicolumn{1}{c}{\upto{n}} 
        &  \multicolumn{1}{c}{} \\  \cline{2-2}
        \nprows &    \\ 
        \cline{2-2}
        \prows & 
        \Rirr \DD_{\prows  \upto{n}} &   \\ \cline{2-2}    
    \end{array}     
\\
    \RR \MM 
    \;
    &= 
    \;
    \DD \CC
    \;
    \equiv 
    \;  
    \begin{array}{l |c|c|}
        \multicolumn{1}{c}{}& \multicolumn{1}{c}{\pcols} & \multicolumn{1}{c}{\npcols}  \\  \cline{2-3}
        \multirow{2}{*}{ $\upto{m}$ } &
        \multirow{2}{*}{ $\DD_{\upto{m}  \pcols} \CC_{\pcols \pcols}$ }
         &
        \\   
         &   &   \\ \cline{2-3}     
    \end{array}      
\end{align*}
\end{proposition}

\section{Connections to order theory}
\label{sec:asideonordertheory}

This result, due to Birkhoff \cite{birkhoff1973lattice}, states that every pair of poset maps $F: \upto{p} \to \lata$, $G: \upto{q} \to \lata$ into a modular order lattice $\lata$ extend to a lattice homomorphism $H: \latb \to \lata$, where $\latb$ is the free distributive lattice generated by $\upto{p}$ and $\upto{q}$.  Concretely, $\latb$ can be realized as the lattice of down-closed subsets of the product poset $\upto{p} \times \upto{q}$.  It turns out that the nonzero entries of $\MM$ correspond exactly to the indices $(s,t)$ such that $H(\upto{s} \times \upto{t}) > H(\upto{s} \times \upto{t} - \{(s,t) \})$, when $\lata$ is the subspace lattice of $\field^m$, $F_s$ is the subspace of $\field^m$ consisting of vectors supported on $\upto{s}$, and $G_t$ is the column space of $D_{\upto{m},\upto{t}}$ \cite{grandis2012homological}. This unassuming fact has proved useful in extending the  ideas of persistent homology from the setting of linear maps and vector spaces to more general algebraic settings \cite{henselman2019decomposition}.

\section{Short-circuit techniques for acceleration and sparsification}
\label{sec:earlystopping}

Many applications of U-match decomposition  make use of both a domain \COB{} $\CC$ and a codomain \COB{} $\RR$.  However, a substantial subset of these applications do not rely on the assumption that $\RR$ is the \underline{\smash {specific}} \COB{} corresponding to $\CC$ (if the decomposition is proper), or even that  $\RR \MM = \DD \CC$.  Rather, in these cases it suffices to assume that there \underline{\smash{exist}} U-match decompositions $\hat \RR \MM = \DD \CC$ and $\RR \MM = \DD \hat \CC$ for \underline{\smash{some}} upper unitriangular $\hat \RR$ and $\hat \CC$.  Such is the case, for example, in  persistent (co)homology computations that require cycle representatives in both persistent homology and persistent cohomology, but not a specific correspondence between the two.\footnote{Of course, in other applications it may be vitally important to understand the correspondence between the two, and in such cases one should probably bite the bullet and requite $\RR \MM = \DD \CC$.}

In such cases, it can be advantageous, computationally, to obtain a $\CC$, which is as sparse as possible. Sparsification encompasses a challenging class of problems in matrix algebra, generally. However, the following observation provides several highly practical heuristics:  recall from Corollary \ref{cor:1blockdiff} that any two proper domain \COB{}s associated to the same mapping array $\DD$ can differ from one another only in the columns indexed by $\VAL(\IMatch) = \{\pcol_1, \ldots, \pcol_k\}$, i.e., only in pivot columns.  In reality, sparsifying $\CC$ therefore means sparsifying pivot columns, since each non-pivot column is uniquely determined.

\begin{lemma}
\label{lem:iffRcol}
Let $\pcol_p$ be a pivot column index, $v$ be a column vector, and $\CC$ be a U-match column operation matrix.  Then swapping column $\pcol_p$ of $\CC$ with $v$ results in a new column operation matrix iff the following conditions hold:
    \begin{enumerate}
        \item $v[\pcol_p] = 1$ and $v[i] = 0$ for $i > \pcol_p$, and
        \item $(\DD v)[i] = 0$ for $i > \pcolp_p$
    \end{enumerate}
\end{lemma}

The heuristics afforded by Lemma \ref{lem:iffRcol} can be described as follows.  Suppose we need to calculate a vector $v$ which is the $\pcol_p$th column of a column-operation matrix associated to $\DD$, and that we have computed $\Ri$ and saved $\Rirr$, as per the proposed compression scheme.  By Proposition \ref{thm:umatchblockidentities}, we can compute column $\pcol_p$ of the column operation matrix that corresponds to $\Ri$ via $\col_{\pcol_p}(\AA^{-1}) \cdot \MM[\pcolp_p, \pcol_p]$.

If we have not saved $\AA^{-1}$ to memory, then we can recover this column by back-substitution. This process entails a sequence of vectors $v_0, \ldots, v_k$, where $v_0$ is the $p$th standard unit vector $e_p = (0, \ldots, 1, \ldots, 0)$, $v_k$ satisfies $\AA v_k = e_p$, and each $v_{t+1}$ differs from $v_t$ by adding at most one nonzero coefficient. If, for any $t$, the vector $v_t$ satisfies the criterion of Lemma \ref{lem:iffRcol}, then we may stop the process early; $v_t$ is already a serviceable column vector, and may have fewer nonzero entries than $\col_{\pcol_p}(\AA^{-1})$.
     
On the other hand, if we already have saved $\AA^{-1}$ to memory, then we may return $v = \col_{\pcol_p}(\AA^{-1}) \cdot \MM[\pcolp_p, \pcol_p]$ directly.  However, we may also delete any nonzero coefficient $v[\pcol_i]$ for which $i<p$ and  $\DD[i, \pcol_q] = 0$ for $i > \pcolp_p$, since deletion of such coefficients will still result in a column vector $v$ which satisfies the criteria of Lemma \ref{lem:iffRcol}.
    
    As a special case of the preceding two optimizations, we may take $v$ to be the standard unit vector whenever $\DD[\pcolp_p, \pcol_p]$ is the lowest nonzero entry in $\DD$.  Where this condition is satisfied, the solution is zero-cost globally optimal.

\medskip

In practice, early stopping, deletion, and zero-cost global optima are all highly relevant to performant persistent (co)homology computation, since it has been observed empirically that the overwhelming majority of pivot columns satisfy the condition for a zero-cost global optimum, in many applied settings \cite{henselman2016matroid, zhang2019hypha, bauer2019ripser, lampret2020chain}.

\section{Lazy access to Jordan bases: alternative approaches}
\label{sec:lazy_jordan_alt}

 U-match factorization provides at least three distinct lazy  approaches to compute a filtered Jordan basis of the filtered differential operator $\DD$: 
 
 \begin{itemize}[align=left,style=nextline,leftmargin=*,font=\normalfont]
 \item[\textbf{Strategy 1:}] Apply Algorithm \ref{alg:revised_lrdec} to obtain the invertible submatrix $\Rirr$ corresponding to a proper U-match decomposition $\RR \MM = \DD \CC$.  The Jordan basis can be constructed directly from $\CC$, as per Theorem \ref{thm:jordanfromumatch}.  Apply the  methods from \S\ref{sec:lazyumatch} to access the columns of $\CC$ in a lazy fashion.

\item[\textbf{Strategy 2:}] Apply Strategy 1 with the following modification: instead of  constructing the columns of $\CC$ exactly, use the early stopping criterion from Appendix \ref{sec:earlystopping} to construct the columns of a \underline{\smash{possibly different}} matrix $\stilde \CC$.  Matrix $\stilde \CC$ is a domain \COB{} for some U-match decomposition $\stilde \RR \MM = \DD \stilde \CC$, and, as such, Theorem \ref{thm:jordanfromumatch} still applies.  We can therefore construct a Jordan basis from the columns of $\stilde \CC$ using that formula.

\item[\textbf{Strategy 3:}] Apply Algorithm \ref{alg:revised_lrdec} to the anti-transposed matrix $\DD^\perp$.  Doing so is nearly equivalent to performing the standard persistent homology column  algorithm \cite{de2011dualities} on $\DD$; the only functional difference concerns the amount of data that one retains or deletes at each step of the elimination process. Informally, this process is the natural counterpart to Algorithm \ref{alg_lrdec} in which one adds columns left to right, rather than adding rows from bottom to top. One can use the resulting U-match decomposition to obtain a Jordan basis, as per Theorem \ref{thm:jordanfromumatch}.  However, in this case \underline{\smash{the early stopping strategy does not}} \underline{\smash{apply}} since the process used to extract the necessary column vectors from the U-match decomposition involves only re-indexing of column vectors, and insertion of a few entries of 0 or 1, and thus, no clearing operations.
 \end{itemize}

If we write $\EE_1$, $\EE_2$, and $\EE_3$ for the bases produced by strategies 1, 2, and 3, respectively, then no two of these arrays must necessarily  equate.  This fact is simplest to observe in the case of $\EE_1$ versus $\EE_2$, since the early stopping procedure produces strictly sparser matrices by design.  

To see how $\EE_1$ and $\EE_3$ may come to disagree, consider the simpler case where $\DD$ is not a boundary matrix but an invertible upper triangular array.  In this case, the matrix $\Rirr = \Ri$ returned by Algorithm \ref{alg:revised_lrdec} will be an identity matrix; the corresponding U-match decomposition will be 
$$
I \MM = \DD \DD^{-1}.
$$
By contrast, if we apply Algorithm \ref{alg:revised_lrdec} to $\DD^{\perp}$ then we obtain $ I \MM = (\DD^{\perp})(\DD^{\perp})^{-1}$; taking anti-transposes and rearranging factors, we then obtain a distinct U-match decomposition, 
$$
\DD \MM = \DD I.
$$
In particular, the first approach yields domain and codomain \COB{}s equal to $\DD^{-1}$ and $I$, respectively; the second approach yields  $I$ and $\DD$, respectively.  This motivating example can be turned into a real example by constructing a chain complex that vanishes outside dimensions 1 and 2, whose boundary operator $\partial_2: \Chains_2 \to \Chains_1$ is given by $\DD$.

\begin{remark}
This entire discussion has a natural dual, under anti-transpose symmetry.  The overall flavor is similar to that of the duality described in Theorem \ref{thm:jordanfromumatch}.  
\end{remark}

\end{document}